\documentclass[smallextended]{svjour3}
\usepackage{amsmath,amssymb,amsfonts,mathtools,comment}
\usepackage{tikz}
\usepackage{algorithm}
\usepackage{algpseudocode}
\usepackage{graphicx,epsfig,color,scalerel,subfigure,booktabs}
\usepackage{siunitx}
\usepackage{afterpage}
\usepackage{enumitem}
\usepackage{upgreek}
\usepackage{mathbbol}
\DeclareSymbolFontAlphabet{\amsmathbb}{AMSb}%
\usepackage{bbold}
\usepackage{acronym}
\usepackage{multirow}
\usepackage{siunitx}
\usepackage{ulem}

\smartqed

\tikzset{shift entire picture/.style n args={2}{execute at end picture={
\pgfmathtruncatemacro{\tmpx}{sign(#1)}
\pgfmathtruncatemacro{\tmpy}{sign(#2)}
\ifnum\tmpx=1
  \ifnum\tmpy=1
   \path[use as bounding box] ([xshift=-#1,yshift=-#2]current bounding box.south west) rectangle 
(current bounding box.north east);
  \else
   \path[use as bounding box] ([xshift=-#1]current bounding box.south west) rectangle 
([yshift=-#2]current bounding box.north east);
  \fi
\else  
  \ifnum\tmpy=1
   \path[use as bounding box] ([yshift=-#2]current bounding box.south west) rectangle 
([xshift=-#1]current bounding box.north east);
  \else
   \path[use as bounding box] (current bounding box.south west) rectangle 
([xshift=-#1,yshift=-#2]current bounding box.north east); 
  \fi
\fi}}}

\definecolor{lightgreen}{rgb}{0.2,0.6,0.2}
\definecolor{lightblue}{rgb}{0.15,0.15,0.85}
\definecolor{darkred}{rgb}{0.85,0.15,0.15}
\definecolor{amber}{rgb}{1.0, 0.49, 0.0}

\newcommand{\aer}[1]{{\leavevmode\color{magenta}{#1}}}

\usepackage[colorlinks=false,breaklinks=true,linkcolor=lightgreen,citecolor=lightblue]{hyperref}

\newcommand{\bn}{\boldsymbol{n}}

\newcommand{\bu}{\boldsymbol{u}}
\newcommand{\bv}{\boldsymbol{v}}

\newcommand{\fb}{\boldsymbol{f}}

\newcommand{\bt}{\boldsymbol{t}}
\newcommand{\bx}{\boldsymbol{x}}

\newcommand{\bsigma}{\boldsymbol{\sigma}}

\newcommand{\bzero}{\boldsymbol{0}}

\newcommand{\bnabla}{\boldsymbol{\nabla}}

\newcommand{\balpha}{\boldsymbol{\alpha}}

\newcommand{\rH}{\mathrm{H}}
\newcommand{\rL}{\mathrm{L}}
\newcommand{\rM}{\mathrm{M}}
\newcommand{\rP}{\mathrm{P}}

\newcommand{\rI}{\mathrm{I}}

\newcommand{\rF}{\mathrm{F}}
\newcommand{\rG}{\mathrm{G}}

\newcommand{\bH}{\mathbf{H}}

\newcommand{\bL}{\mathbf{L}}

\newcommand{\bZ}{\mathbf{Z}}
\newcommand{\bA}{\mathbf{A}}
\newcommand{\bB}{\mathbf{B}}
\newcommand{\bC}{\mathbf{C}}

\newcommand{\bbI}{\mathbb{I}}
\newcommand{\bbR}{\mathbb{R}}

\newcommand{\cA}{\mathcal{A}}
\newcommand{\cT}{\mathcal{T}}

\newcommand{\cE}{\mathcal{E}}

\newcommand{\vdiv}{\mathop{\mathrm{div}}\nolimits}

\newcommand\bcurl{\mathop{\mathbf{curl}}\nolimits}
\newcommand\vcurl{\mathop{\mathrm{curl}}\nolimits}

\newcommand\bDelta{\boldsymbol{\Delta}}

\acrodef{pde}[PDE]{Partial Differential Equation} 
\acrodefplural{pde}[PDEs]{Partial Differential Equations} 
\acrodef{dof}[DoF]{Degree of Freedom} 
\acrodefplural{dof}[DoFs]{Degrees of Freedom} 
\acrodef{fe}[FE]{Finite Element} 
\acrodefplural{fe}[FEs]{Finite Elements} 
\acrodef{vem}[VEM]{Virtual Element Method}
\acrodefplural{vem}[VEMs]{Virtual Element Methods}

\numberwithin{remark}{section}
\allowdisplaybreaks

\title{Stream function -- pressure virtual element methods for the Stokes--Darcy interface problem}

\author{
Franco Dassi
\and
Rekha Khot
\and
David Mora
\and
Andr\'es E. Rubiano
\and  
Ricardo Ruiz-Baier
}
\institute{ F. Dassi \at
                Dipartimento di Matematica e Applicazioni, Università degli studi di Milano Bicocca, Via Roberto Cozzi 55, 20125, Milano, Italy. \email{Franco.Dassi@unimib.it}
            \and
            R. Khot \at
                Department of Mathematics, Indian Institute of Technology Palakkad,  Kanjikode 678623, Kerala, India.  \email{RekhaKhot@iitpkd.ac.in}
            \and
            D. Mora \at
              GIMNAP, Departamento de Matemática, Universidad del Bío-Bío, 4051381 Concepción, Chile; and CI$^2$MA, Universidad de Concepción, Concepción, Chile. 
              \email{DMora@ubiobio.cl}
            \and
            A. E. Rubiano \and R. Ruiz-Baier \at
                School of Mathematics, Monash University, 9 Rainforest Walk, Melbourne, VIC 3800, Australia.  \email{Andres.RubianoMartinez@monash.edu,
                Ricardo.RuizBaier@monash.edu}            
}

\date{Received: date / Accepted: date}

\begin{document}

\maketitle

\begin{abstract}
    This paper introduces a novel   Virtual Element Method (VEM) for the coupled Stokes--Darcy system in primal-primal form. In the free-flow Stokes domain, we implement a stream function formulation that inherently satisfies the incompressibility constraint and reduces computational cost. Across the  interface, mass conservation, normal stress balance, and the Beavers--Joseph--Saffman slip condition are   enforced to couple the biharmonic stream function equation with the Darcy's pressure equation. Leveraging VEM's  ability to handle   general polygonal meshes, the proposed method naturally accommodates   irregular interface geometries without requiring remeshing or adaptive refinement. The accuracy of the method is validated through several numerical simulations that include applications to dead-end filtration,  and network flow in bioartificial organs.
\end{abstract}

%%%%%%%%%%%%%%%%%%%%%%%%%%%%%%%%%%%%%%%%%%%%%%%%%
\section{Introduction}
The coupling of free-flow with porous media flow is a fundamental challenge in fluid mechanics, with several applications ranging from groundwater hydrology and reservoir engineering to industrial filtration and biomedical modelling (see e.g. \cite{Bukac2024,dz2011,discacciati2002mathematical,Hanspal2006,Qi2025}). 
%Governed by the coupled Stokes--Darcy equations, the model requires a proper treatment of the fluid--porous interface. 
The interfacial coupling is achieved through continuity of normal velocities (mass conservation), balance of normal stresses, and the Beavers--Joseph--Saffman tangential slip condition \cite{Beavers_Joseph_1967,s1971}. While more general frameworks account for porous media with cracks, and the incorporation of other linear and
nonlinear equations in the coupled problem, such as Navier--Stokes, Brinkman and Forchheimer (see, e.g., \cite{bernardi2005coupling,cimolin2013navier,ervin2009coupled,gatica2012twofold,xie2008uniformly}), the classical Stokes--Darcy system maintains the essential features of the interface mechanisms and remains the most important benchmark for modelling fluid-porous interfaces.

A wide range of primal and mixed \acp{fe}, discontinuous Galerkin (dG), and Hybrid High--Order (HHO) methods have been developed for the Stokes--Darcy problem; a non-exhaustive list of representative references includes \cite{anaya2019,antonietti26,Badea2010,boon22,bukac2023,burman2021,CAMANO2015362,edr2022,gos2011,KANSCHAT2010,vexler2026fully,Wen2020}. Most of these approaches rely on velocity--pressure, velocity--pressure--pseudostress, or vorticity--pressure formulations in the Stokes region and either primal or mixed formulations in the Darcy region. These formulations require the design of compatible discrete spaces that satisfy appropriate inf--sup stability conditions and accurately enforce the interface coupling conditions. While conforming finite element methods often employ matching discretisations across the interface, dG and HHO formulations can naturally accommodate non-matching meshes through the weak enforcement of the coupling conditions.

Motivated by the need for geometric flexibility, recent years have also witnessed the development of \acp{vem} for the Stokes--Darcy problem on general polygonal meshes \cite{Liu2019,Mishra2024,WANG2019998}. Originally introduced in \cite{bbcmmr13}, the \ac{vem}  extends the finite element paradigm to general polygonal and polyhedral meshes while retaining many of its fundamental theoretical and computational properties. This flexibility makes \ac{vem} particularly attractive for interface problems, where complex geometries and non-matching meshes frequently arise.

In two-dimensional simply connected domains, however, the incompressibility constraint can be identically satisfied by introducing a scalar stream function \cite{gr1986}, thereby eliminating the pressure variable from the Stokes equations. While this reformulation has been widely successfully exploited within the FE and dG frameworks \cite{camano2026mixed,carstensen21,cayco86,faisal03,kim21,Liu00}, its adaptation to virtual elements is more recent. In addition to reducing the number of \acp{dof}, this well-known reformulation bypasses the need for inf–sup stable velocity–pressure pairs and yields a fourth-order problem for the stream function. Because the resulting formulation inherently requires $\rH^2(\Omega)$ conforming approximation spaces, $C^1$-conforming VEMs are particularly well-suited, and several stream function formulations and associated pressure recovery algorithms have been successfully established for   uncoupled Stokes and Navier--Stokes settings \cite{ams2024,abmv,mrs2022,ms2025}. However, their extension to the coupled Stokes--Darcy problem remains largely unexplored. This is precisely the gap we  bridge in this paper.

\paragraph{Main contributions.} The key novelties and contributions of this paper are summarised below:
\begin{itemize}
    \item We develop a continuous analysis for the Stokes--Darcy interface problem in stream function -- pressure formulation. This structure does not require inf-sup stability.
    \item We design a $C^1$--$C^0$ lowest-order VEM discretisation for the stream function -- pressure coupled Stokes--Darcy system. Such scheme involves only $3N_{\mathrm{S}}+N_{\mathrm{D}}+4N_{\Sigma}$ DoFs, where $N_{\mathrm{S}}$, $N_{\mathrm{D}}$, and $N_{\Sigma}$ are the number of interior vertices in the Stokes domain, Darcy domain, and on the interface, respectively. 
    \item We conduct the a priori error analysis for the proposed VEM discretisation.
    \item We incorporate a new functionality in the open source object-oriented library \texttt{vem++} \cite{dassi2023vem++}, allowing a straightforward implementation of interface problems through the class \texttt{vem::vemMesh2dDofFilter}.
    \item We provide several numerical experiments, confirming the theoretical findings and the applicability of the method.
\end{itemize}

\paragraph{Plan of the paper.} This paper is organised in the following manner.  The remainder of this section introduces preliminary notation and defines some classical functional spaces. Section~\ref{sec:formulation} introduces the stream function -- pressure formulation for the Stokes--Darcy interface problem. Next, the \ac{vem} spaces, \acp{dof}, polynomial projection operators, interpolation estimates, and the discrete problem are provided in Section~\ref{sec:vem}. Section~\ref{sec:apriori} derives a  a priori error estimates for the discrete scheme. We show some numerical experiments in Section~\ref{sec:numerical_examples} to confirm our theoretical findings and Section~\ref{sec:conclusion} provides a summary of our conclusions and remaining open challenges.  

\paragraph{Notation and preliminaries.} 
Throughout this paper, $\Omega:=\Omega_{\mathrm{S}}\cup \Omega_{\mathrm{D}} \subset \mathbb{R}^2$ is assumed to be a polygonal domain with Lipschitz continuous boundary $\Gamma := \partial \Omega$. For a vector function $\bsigma:\Omega \to \bbR^2$ and a scalar function $\phi:\Omega \to \bbR$, we define the vector gradient $\bnabla \bsigma:\Omega \to \bbR^{2\times 2}$, the scalar gradient $\nabla \phi:\Omega \to \bbR^2$, the vector divergence $\vdiv\bsigma : \Omega \to \bbR$, the vector rotational $\vcurl \bsigma:\Omega \to \bbR$, the scalar rotational $\bcurl \phi:\Omega \to \bbR^2$, the vector Laplacian $\bDelta \bsigma:\Omega \to \bbR^2$, the scalar Laplacian $\Delta \phi:\Omega \to \bbR$, and the Hessian matrix $\nabla^2 \phi: \Omega \to \bbR^{2\times 2}$ as $ (\bnabla \bsigma)_{ij} := \partial_j \bsigma_i$, $ (\nabla \phi)_{i} := \partial_i \phi$, $\vdiv \bsigma := \sum_i   \partial_i \bsigma_i$, $\vcurl \bsigma := \partial_1 \bsigma_2 - \partial_2 \bsigma_1$, $\bcurl \phi := (\partial_2 \phi, -\partial_1 \phi)^{\tt t}$, $(\bDelta\bu)_i := \vdiv(\bnabla\bu)_i$, $\Delta\phi := \vdiv(\nabla\phi)$, and $\nabla^2\phi := \bnabla(\nabla\phi)$, respectively. Moreover, the following identities hold
\begin{gather*} %\label{eq:differential_identities}
    \vdiv(\bcurl(\phi)) = 0, \quad \bcurl(\phi)\cdot \bn= \nabla \phi \cdot \bt, \quad \text{and} \quad \bcurl(\phi)\cdot \bt = -\nabla \phi \cdot \bn,
\end{gather*}
with $\bn$ and $\bt$ denoting the respective unit outward normal and tangential vectors with respect to $\Gamma$.
 
Given $s \geq 0$, we denote the usual Hilbert space of scalar and vector functions with domain $\Omega$ by $\rH^{s}(\Omega)$ and $\bH^s(\Omega)$. The usual norm (resp. semi-norm) of $\rH^{s}(\Omega)$ is denoted by $\|\bullet\|_{s,\Omega}$ (resp. $|\bullet|_{s,\Omega}$). When $s=0$, we define the Hilbert space $\rH^0(\Omega):=\rL^2(\Omega)$ with the usual inner product $(\bullet, \bullet)_\Omega$ (similarly for $\bL^2(\Omega)$), while $\rL^2_0(\Omega)$ denotes the subspace of $\rL^2(\Omega)$ consisting of functions with zero mean over $\Omega$. The trace space of $\rH^1(\Omega)$ is denoted by $\rH^{\frac{1}{2}}(\Gamma)$, and its dual space by $\rH^{-\frac{1}{2}}(\Gamma)$. The notation $\langle \bullet, \bullet \rangle_{\Gamma}$ represents the duality pairing between $\rH^{-\frac{1}{2}}(\Gamma)$ and $\rH^{\frac{1}{2}}(\Gamma)$, induced by the $\rL^2(\Gamma)$ inner product.

Finally, $C$ denotes a generic positive constant (which can have different values at its different occurrences) independent of the mesh size $h$ and physical constants. Moreover, the notation $A \lesssim B$ indicates that for $A,B\geq0$ we have that $A \leq C B$.

%%%%%%%%%%%%%%%%%%%%%%%%%%%%%%%%%%%%%%%%%%%%%%%%%
\section{The model problem}\label{sec:formulation}
In this section we introduce the stream function -- pressure formulation for the Stokes--Darcy interface problem. The strong and weak formulations are derived using identities involving the $\bcurl$ operator. Finally, the well-posedness result is stated.

\subsection{The classical Stokes--Darcy interface problem}
Let $\Omega_\mathrm{S}$ be a simply connected free-flow domain filled with a viscous incompressible fluid whose dynamics is governed by the velocity--pressure form of Stokes' equations 
\begin{subequations} \label{eq:S}
    \begin{alignat}{2}
        - \mu \bDelta \bu + \nabla p &= \fb,  &&\quad \text{in } \Omega_\mathrm{S}, \label{eq:S1} \\
         \vdiv \bu &= 0, &&\quad \text{in } \Omega_\mathrm{S}, \label{eq:S2}
    \end{alignat}
\end{subequations}
where $\bu:\Omega_\mathrm{S} \to \bbR^2$ denotes the velocity of the fluid and $p:\Omega_\mathrm{S} \to \bbR$ the fluid pressure, both in the free-flow domain. In addition, $\mu>0$ is the fluid viscosity and $\fb:\Omega_\mathrm{S} \to \bbR^2$ is a source/sink term. The boundary $\partial \Omega_\mathrm{S}$ is the wall of the container separated into $\Gamma_\mathrm{S}$ and $\Sigma$ such that $\partial \Omega_\mathrm{S}=\Gamma_\mathrm{S}\cup\Sigma$ and $\Sigma=\partial \Omega_\mathrm{S}\setminus \Gamma_\mathrm{S}$. We consider standard Dirichlet boundary conditions on $\Gamma_\mathrm{S}$ as follows
\begin{alignat}{2} \label{eq:SBC}
    \bu &= \bzero, && \quad \text{on } \Gamma_\mathrm{S}. 
\end{alignat}
On the other hand, consider a porous media $\Omega_\mathrm{D}$ filled with fluid in the intrinsic pore space with %velocity $\bzeta:\Omega_\mathrm{D} \to \bbR^2$ and 
pressure $\varphi:\Omega_\mathrm{D} \to \bbR$ satisfying Darcy's Law
\begin{alignat}{2}\label{eq:D}
    %\bzeta + K \nabla \varphi &= \bzero,  &&\quad \text{in } \Omega_\mathrm{D}, \label{eq:D1} \\
    %\vdiv \bzeta &= g, &&\quad \text{in } \Omega_\mathrm{D}. \label{eq:D2}
    \vdiv (-\kappa \nabla \varphi) &= g, &&\quad \text{in } \Omega_\mathrm{D}. 
\end{alignat}
Here, $\kappa>0$ is the hydraulic conductivity and $g:\Omega_\mathrm{D}\to \bbR$ is another source/sink term and we assume that
$\int_{\Omega_\mathrm{D}}g=0$. Similarly, the boundary $\partial \Omega_\mathrm{D}$ fulfils the splitting $\Gamma_\mathrm{D}\cup\Sigma$, with $\Sigma=\partial \Omega_\mathrm{D}\setminus \Gamma_\mathrm{D}$. In this case, no flux is prescribed in the following way
\begin{alignat}{2} \label{eq:DBC}
    (-\kappa \nabla \varphi) \cdot \bn_\mathrm{D} &= 0, && \quad \text{on } \Gamma_\mathrm{D},
\end{alignat}
where $\bn_\mathrm{D}$ is the unit outward normal vector with respect to $\Gamma_\mathrm{D}$. Finally, we focus on the remaining portion of the boundary, the interface $\Sigma:=\partial\Omega_\mathrm{S}\cap\partial\Omega_\mathrm{D}$, where the transmission conditions are  
\begin{subequations} \label{eq:TC}
    \begin{alignat}{2}
        \bu \cdot \bn_\Sigma &= (-\kappa \nabla \varphi) \cdot \bn_\Sigma, &&\quad \text{on } \Sigma, \label{eq:TC1} \\
        -  (\mu \bnabla\bu - p\bbI)\bn_\Sigma \cdot \bn_\Sigma &= \varphi, &&\quad \text{on } \Sigma, \label{eq:TC2} \\
        -  (\mu \bnabla\bu - p\bbI)\bn_\Sigma \cdot \bt_\Sigma &= \alpha\frac{\mu }{\sqrt{\kappa}} \bu \cdot \bt_\Sigma, &&\quad \text{on } \Sigma. \label{eq:TC3}
    \end{alignat}
\end{subequations}
The third equation \eqref{eq:TC3} above is known as the Beavers--Joseph--Saffman transmission condition, in which the parameter $\alpha$ refers to the slip rate coefficient, the quotient $\mu/\sqrt{\kappa}$ represents the effective interfacial friction coefficient and $\bbI\in\bbR^{2\times2}$ is the identity matrix. In particular, we fix $\bn_\Sigma$ as the outward (resp. inward) unit normal vector with respect to $\partial \Omega_\mathrm{S}$ (resp. $\Omega_\mathrm{D}$) when seen on $\Sigma$. A sketch of the computational domain highlighting the boundary configuration and interface is shown in Fig.~\ref{fig:computational_domain}. 

\begin{figure}[h!]
    \centering
    \subfigure[Single boundary condition.\label{fig:computational_domain}]{\includegraphics[width=0.34\textwidth,trim={0.cm 0.cm 13.4cm 0.cm},clip]{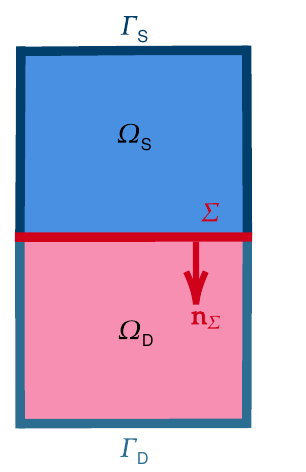}}
    \subfigure[Mixed boundary condition.\label{fig:domain_mixed_case}]
    {\includegraphics[width=0.34\textwidth,trim={0.cm 0.cm 13.4cm 0.cm},clip]{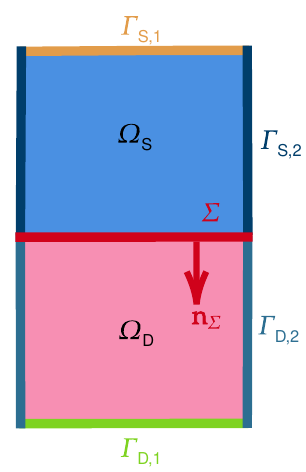}}
    \caption{Sketch of the computational domain and possible boundary configuration.}
    \label{fig:computational_domain_complete}
\end{figure}
\begin{remark} \label{rem:BCAlt}

It is possible to consider various combinations of boundary conditions on $\Omega_\mathrm{S}$ and $\Omega_\mathrm{D}$, provided they are compatible with the transmission conditions  \eqref{eq:TC}. For example, one can impose full Neumann conditions on $\Gamma_\mathrm{S}$ and full Dirichlet conditions on $\Gamma_\mathrm{D}$, namely
\begin{subequations}\label{eq:BCAlt}
    \begin{alignat}{2}
    (\mu\nabla\bu - p\bbI)\bn_\mathrm{S} &= \bzero, &&\quad \text{on } \Gamma_\mathrm{S},\label{eq:BCAltS}\\
    \varphi &= 0, &&\quad \text{on } \Gamma_\mathrm{D},\label{eq:BCAltD}
    \end{alignat}
\end{subequations}
which is consistent with \eqref{eq:TC2}. More generally, non--homogeneous and/or mixed boundary conditions may also be considered, as long as the conditions \eqref{eq:SBC} and \eqref{eq:DBC} (or alternatively \eqref{eq:BCAltS} and \eqref{eq:BCAltD}) are compatible. Some examples are given below: for $z_\mathrm{D,2}\in\rH^{-\frac{1}{2}}(\Gamma_{\mathrm{S},2}\cup\Gamma_{\mathrm{D},2})$ and $\varphi_\mathrm{D,2}\in\rH^{\frac{1}{2}}(\Gamma_{\mathrm{S},2}\cup\Gamma_{\mathrm{D},2})$, Fig.~\ref{fig:domain_mixed_case} illustrates the following configuration of the boundary sections $\Gamma_{\mathrm{S},1}$, $\Gamma_{\mathrm{S},2}$, $\Gamma_{\mathrm{D},1}$, and $\Gamma_{\mathrm{D},2}$
%\begin{equation*}
\begin{minipage}{0.48\textwidth}
\begin{alignat*}{2}
(\mu\nabla\bu - p\bbI)\bn_\mathrm{S,1}&= \bzero, &&\quad \text{on } \Gamma_{\mathrm{S},1},\\
\bu&= z_\mathrm{D,2} \bn_\mathrm{S,2}, &&\quad \text{on } \Gamma_{\mathrm{S},2},\\
\varphi&=0, &&\quad \text{on } \Gamma_{\mathrm{D},1},\\
(-\kappa\nabla \varphi)\cdot \bn_\mathrm{D,2}&= -z_\mathrm{D,2}, &&\quad \text{on } \Gamma_{\mathrm{D},2}.
\end{alignat*}
\centering \text{(Case 1)}
\end{minipage}
\hfill
\begin{minipage}{0.48\textwidth}
\begin{alignat*}{2}
\bu&= \bzero, &&\quad \text{on } \Gamma_{\mathrm{S},1},\\
(\mu\nabla\bu - p\bbI)\bn_\mathrm{S,2}&= \varphi_\mathrm{D,2}\bn_\mathrm{S,2}, &&\quad \text{on } \Gamma_{\mathrm{S},2},\\
(-\kappa\nabla \varphi)\cdot \bn_\mathrm{D,1}&= 0, &&\quad \text{on } \Gamma_{\mathrm{D},1},\\
\varphi&=-\varphi_\mathrm{D,2}, &&\quad \text{on } \Gamma_{\mathrm{D},2}.
\end{alignat*}
\centering \text{(Case 2)}
\end{minipage}
%\end{equation*}
\end{remark}

In view of the boundary conditions \eqref{eq:SBC} and \eqref{eq:DBC}, we define the following Hilbertian functional space
\begin{equation*}
    \bH^1_{\Gamma_\mathrm{S}}(\Omega_\mathrm{S}):= \{ \bv\in \bH^1(\Omega_\mathrm{S}): \bv= \bzero \text{ on } \Gamma_\mathrm{S}\}.%\quad \text{and} \quad \rH^1(\Omega_\mathrm{D}):= \{  \psi\in \rH^1(\Omega_\mathrm{D}): \int_{\Omega_\mathrm{D}}\psi= 0 \}.
\end{equation*}
We multiply the governing equations by appropriate test functions, integrate by parts, and apply \eqref{eq:TC} to arrive at the weak formulation: given $\fb\in \bL^2(\Omega_\mathrm{S})$ and %$\btau^\mathrm{S}_2\in \bH^{-\frac{1}{2}}(\Gamma_\mathrm{S}_2)$,
$g\in\rL^2_0(\Omega_\mathrm{D})$,  %and $\varrho^\mathrm{D}_2\in\rH^{-\frac{1}{2}}(\Gamma_\mathrm{D})$
 find $(\bu,p,\varphi)\in \bH^1_{\Gamma_\mathrm{S}}(\Omega_\mathrm{S}) \times \rL^2(\Omega_\mathrm{S}) \times \rH^1(\Omega_\mathrm{D})$ such that 
\begin{subequations}\label{eq:WC}
    \begin{alignat}{2}
    \mu\int_{\Omega_\mathrm{S}} \!\!\bnabla\bu : \bnabla\bv - \! \int_{\Omega_\mathrm{S}} \!p \vdiv\bv + \frac{\alpha\mu}{\sqrt{\kappa}} \int_\Sigma (\bu\cdot\bt_\Sigma) (\bv\cdot\bt_\Sigma) + \int_\Sigma \varphi \bv\cdot\bn_\Sigma  &= \int_{\Omega_\mathrm{S}} \fb \cdot \bv, &&\,\, \forall \bv\in \bH^1_{\Gamma_\mathrm{S}}(\Omega_\mathrm{S}), \label{eq:WC1} \\ %&& \notag \\ 
        %& \quad + \langle\btau^\mathrm{S}_2, \bv\rangle_{\Gamma_\mathrm{S}_2}, &&\quad \forall \bv\in \bH^1_{\Gamma_\mathrm{S}}(\Omega), \label{eq:WC1} \\
        \int_{\Omega_\mathrm{S}} q \vdiv \bu &= 0, &&\,\, \forall q\in \rL^2(\Omega_\mathrm{S}), \label{eq:WC2} \\
        -\kappa\int_{\Omega_\mathrm{D}} \nabla \varphi \cdot \nabla \psi + \int_\Sigma \bu\cdot \bn_\Sigma \psi &= -\int_{\Omega_\mathrm{D}}g\psi,  && \,\, \forall \psi\in \rH^1(\Omega_\mathrm{D}). \label{eq:WC3}
    \end{alignat}
\end{subequations}

\subsection{The stream function -- pressure interface formulation}\label{sec:stream_pressure} Motivated by the incompressibility condition provided in \eqref{eq:S2}, we define the spaces
\begin{gather*}
    \bZ := \{ \bv\in \bH^1_{\Gamma_\mathrm{S}}(\Omega_\mathrm{S}) : \vdiv \bv = 0 \} \quad \text{and} \quad \rH^2_{\Gamma_\mathrm{S}}(\Omega_\mathrm{S}) = \{\xi\in\rH^2(\Omega_\mathrm{S}) : \xi = 0 \text{ and } \nabla\xi\cdot \bn_\mathrm{S} = 0 \text{ on } \Gamma_\mathrm{S}\}.
\end{gather*}
By the stream function representation (see \cite[Chap.~I, Sect.~5.2 and Chap.~IV, Sect.~2.2]{gr1986}), for any $\bu \in \mathbf{Z}$ there exists $\chi \in \mathrm{H}^2(\Omega_\mathrm{S})\setminus \mathbb{R}$ such that $\bu = \bcurl \chi$. This identification allows us to arrive at the following stream function -- pressure weak formulation for the Stokes--Darcy interface problem: given $\fb\in \bL^2(\Omega_\mathrm{S})$ and %$\btau^\mathrm{S}_2\in \bH^{-\frac{1}{2}}(\Gamma_\mathrm{S}_2)$,
$g\in\rL^2_0(\Omega_\mathrm{D})$ % and %$\varrho^\mathrm{D}_2\in\rH^{-\frac{1}{2}}(\Gamma_\mathrm{D})$
; find $(\chi,\varphi)\in \rH^2_{\Gamma_\mathrm{S}}(\Omega_\mathrm{S}) \times \rH^1(\Omega_\mathrm{D})$ such that 
\begin{subequations}\label{eq:WSP}
    \begin{alignat}{2}
        \mu\int_{\Omega_\mathrm{S}} \nabla^2\chi : \nabla^2\xi + \frac{\alpha\mu}{\sqrt{\kappa}} \int_\Sigma (\nabla \chi\cdot\bn_\Sigma) (\nabla \xi\cdot\bn_\Sigma) + \int_\Sigma \varphi \nabla \xi\cdot\bt_\Sigma &= \int_{\Omega_\mathrm{S}} \fb \cdot \bcurl \xi, &&\,\, \forall \xi\in \rH^2_{\Gamma_\mathrm{S}}(\Omega_\mathrm{S}), \label{eq:WSP1} \\ % && \notag \\ 
        %& \quad + \langle\btau^\mathrm{S}_2, \bcurl \xi\rangle_{\Gamma_\mathrm{S}_2}, &&\quad \forall \xi\in \rH^2_{\Gamma_\mathrm{S}}(\Omega), \label{eq:WSP1} \\
        -\kappa\int_{\Omega_\mathrm{D}} \nabla \varphi \cdot \nabla \psi + \int_\Sigma \psi \nabla \chi\cdot\bt_\Sigma &= -\int_{\Omega_\mathrm{D}}g\psi,  && \,\, \forall \psi\in \rH^1(\Omega_\mathrm{D}). \label{eq:WSP2}
    \end{alignat}
\end{subequations}

We now define the bilinear forms $a(\bullet,\bullet):\rH^2_{\Gamma_\mathrm{S}}(\Omega_\mathrm{S})\times \rH^2_{\Gamma_\mathrm{S}}(\Omega_\mathrm{S})\to \bbR$, $b(\bullet,\bullet):\rH^2_{\Gamma_\mathrm{S}}(\Omega_\mathrm{S})\times \rH^1(\Omega_\mathrm{D}) \to \bbR$, and $c(\bullet,\bullet):\rH^1(\Omega_\mathrm{D})\times \rH^1(\Omega_\mathrm{D})\to \bbR$ together with the linear forms $\rF(\bullet):\rH^2_{\Gamma_\mathrm{S}}(\Omega_\mathrm{S})\to \bbR$ and $\rG(\bullet):\rH^1(\Omega_\mathrm{D})\to \bbR$ given by
\begin{gather*}
    a(\chi,\xi) = \langle \bA \chi,\xi \rangle := \mu\int_{\Omega_\mathrm{S}} \nabla^2\chi : \nabla^2\xi + \frac{\alpha\mu}{\sqrt{\kappa}} \int_\Sigma (\nabla \chi\cdot\bn_\Sigma) (\nabla \xi\cdot\bn_\Sigma) =: a^{\nabla^2}(\chi,\xi) + a^{\Sigma}(\chi,\xi),\\
    b(\xi,\psi) = \langle \bB \xi, \psi \rangle := \int_\Sigma \psi \nabla \xi\cdot\bt_\Sigma,\quad
    c(\varphi,\psi) = \langle \bC \varphi, \psi \rangle := \kappa\int_{\Omega_\mathrm{D}} \nabla \varphi \cdot \nabla \psi,\\
    \rF(\xi) := \int_{\Omega_\mathrm{S}} \fb \cdot \bcurl \xi,\quad
    \rG(\psi) := -\int_{\Omega_\mathrm{D}}g\psi,
\end{gather*}
where $\bA:\rH^2_{\Gamma_\mathrm{S}}(\Omega_\mathrm{S})\to[\rH^2_{\Gamma_\mathrm{S}}(\Omega_\mathrm{S})]'$, $\bB:\rH^2_{\Gamma_\mathrm{S}}(\Omega_\mathrm{S})\to [\rH^1(\Omega_\mathrm{D})]'$, and $\bC:\rH^1(\Omega_\mathrm{D})\to [\rH^1(\Omega_\mathrm{D})]'$ are the linear operators induced by $a(\bullet,\bullet)$, $b(\bullet,\bullet)$, and $c(\bullet,\bullet)$, respectively. Note that $\rF\in [\rH^2_{\Gamma_\mathrm{S}}(\Omega_\mathrm{S})]'$ and $\rG\in [\rH^1(\Omega_\mathrm{D})]'$, where the notation $[\rH^2_{\Gamma_\mathrm{S}}(\Omega_\mathrm{S})]'$ (resp. $[\rH^1(\Omega_\mathrm{D})]'$) indicates the dual space of $\rH^2_{\Gamma_\mathrm{S}}(\Omega_\mathrm{S})$ (resp. $\rH^1(\Omega_\mathrm{D})$). These definitions lead to the following perturbed saddle--point problem (equivalent to \eqref{eq:WSP}) written in operator form, i.e., in the dual of the solution space, given by
\begin{equation}\label{eq:operator_form} 
    \begin{pmatrix}
        \bA & \bB^* \\
        \bB & -\bC 
    \end{pmatrix}
    \begin{pmatrix}
        \chi \\ 
        \varphi
    \end{pmatrix}
    = 
    \begin{pmatrix}
        \rF \\ 
        \rG    
    \end{pmatrix} 
    ,\quad \text{in }[\rH^2_{\Gamma_\mathrm{S}}(\Omega_\mathrm{S})\times \rH^1(\Omega_\mathrm{D})]'. 
\end{equation}

\subsection{Well-posedness of the continuous problem}\label{sec:wp}
We now aim to prove the well-posedness of the problem provided in \eqref{eq:operator_form}. First, we collect in the lemma below the properties of the linear operators involved in the formulation.
\begin{lemma}[properties of the continuous operators]\label{lem:properties_linear_ops}
    The following bounds hold:
    \begin{subequations}\label{eq:bounds}
        \begin{alignat}{2}
        |\langle \bA\chi,\xi \rangle| & \lesssim \mu\max\left\{1,\frac{\alpha}{\sqrt{\kappa}}\right\} \|\chi\|_{2,\Omega_\mathrm{S}}\|\xi\|_{2,\Omega_\mathrm{S}}, &&\quad \forall  \chi,\xi \in \rH^2_{\Gamma_\mathrm{S}}(\Omega_\mathrm{S}),\label{bound:a}\\
        \langle \bA\xi,\xi \rangle & \geq \mu \|\xi\|^2_{2,\Omega_\mathrm{S}}, &&\quad \forall  \xi \in \rH^2_{\Gamma_\mathrm{S}}(\Omega_\mathrm{S}),\label{coer:a}\\
        |\langle \bB \xi,\psi\rangle| & \lesssim \|\xi\|_{2,\Omega_\mathrm{S}}\|\psi\|_{1,\Omega_\mathrm{D}}, &&\quad \forall \xi \in \rH^2_{\Gamma_\mathrm{S}}(\Omega_\mathrm{S}),\, \forall\psi \in \rH^1(\Omega_\mathrm{D}),\label{bound:b}\\
        |\langle \bC\varphi,\psi\rangle| & \leq \kappa\|\varphi\|_{1,\Omega_\mathrm{D}}\|\psi\|_{1,\Omega_\mathrm{D}}, &&\quad  \forall \varphi,\psi \in \rH^1(\Omega_\mathrm{D}), \label{bound:c}\\
        %\langle \bC\psi,\psi\rangle & =  \kappa \|\psi\|^2_{1,\Omega_\mathrm{D}}, &&\quad \forall \psi \in \rH^1(\Omega_\mathrm{D}), \label{coer:c}\\
        |\rF(\xi)| & \lesssim \|\fb\|_{0,\Omega_\mathrm{S}}\|\xi\|_{2,\Omega_\mathrm{S}}, &&\quad \forall \xi \in \rH^2_{\Gamma_\mathrm{S}}(\Omega_\mathrm{S}), \label{bound:F}\\
        |\rG(\psi)| & \lesssim \|g\|_{0,\Omega_\mathrm{D}}\|\psi\|_{1,\Omega_\mathrm{D}}, &&\quad \forall \psi \in \rH^1(\Omega_\mathrm{D}).\label{bound:G}
        \end{alignat}
    \end{subequations}
  %  \rek{Is the coercivity constant correct? What about the interface term?} \aer{Coercivity is fine, I do not understand the second question.} \rrb{[I think this refers to what happens with the  $a^\Sigma$ term. For coercivity, we can simply use that  $a^\Sigma(\chi,\chi)$   is non-negative and $a^{\nabla^2}(\chi,\chi)$ suffices to get the coercivity in $H^2_{\Gamma_S}(\Omega_S)$. Another possibility is to redefine the norm for $\chi$ including also the normal derivatives on the interface (and with the weight $\alpha\mu/\sqrt{\kappa}$) so coercivity and boundedness hold with constant 1.]} \aer{Something like the robust paper? This could work but I'm not sure about the bounds for the coupling term $\bB$ and the right hand side (since there the norms are considering both spaces)} \rrb{I think we can leave it as is, but if we modify the norm for $\chi$ then we can simply weight with the inverse coefficient the norm for $\varphi$, and then the boundedness of B will still be fine, since you're bounding from above and can simply add the missing terms to complete the required inequality. \rek{No worries! Let's keep it as it is!!}}
\end{lemma}
\begin{proof}
    Standard arguments involving the Cauchy--Schwarz inequality, the trace inequality (used in \eqref{bound:a}, \eqref{bound:b}, and \eqref{bound:G}), and the Poincar\'e inequality (used in \eqref{bound:F} and \eqref{bound:G}), lead to the result.
    \qed
\end{proof}

Following \cite{gos2011}, the kernel of the system \eqref{eq:operator_form} is characterised next.
\begin{lemma}[kernel characterisation]\label{lem:ker}
    Let $(\chi,\varphi)\in \rH^2_{\Gamma_\mathrm{S}}(\Omega_\mathrm{S})\times \rH^1(\Omega_\mathrm{D})$ be a solution to the homogeneous system in \eqref{eq:operator_form}. Then, there exists $\lambda\in \bbR$ such that
    $$(\chi,\varphi) = (0,\lambda).$$
\end{lemma}
\begin{proof}
    Starting from \eqref{eq:operator_form}, we can rewrite the problem in operator form as
    $$\bA \chi + \bC \varphi = 0 \quad \text{in} \quad [\rH^2_{\Gamma_\mathrm{S}}(\Omega_\mathrm{S})\times \rH^1(\Omega_\mathrm{D})]'.$$
    Testing this equation with $(\chi,\varphi)\in \rH^2_{\Gamma_\mathrm{S}}(\Omega_\mathrm{S})\times \rH^1(\Omega_\mathrm{D})$ and using the coercivity of $\bA$ (cf. \eqref{coer:a}) together with the definition of $\bC$, we obtain
    $$\mu\|\chi\|_{2,\Omega_\mathrm{S}}^{2} + \kappa|\varphi|_{1,\Omega_\mathrm{D}}^{2}\leq 0.$$
    Therefore, we deduce that $\chi=0$ and  $\varphi=\lambda$ for some $\lambda\in \bbR$.
    \qed
\end{proof}
\begin{figure}[h!]
    \centering
    \subfigure[Immersed Stokes domain.\label{fig:domain_stokes_in_darcy}]{\includegraphics[width=0.34\textwidth,trim={0.cm 0.cm 13.4cm 0.cm},clip]{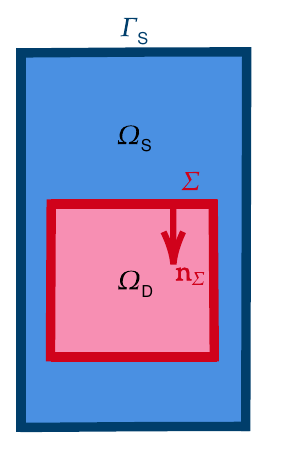}}
    \subfigure[Immersed Darcy domain.\label{fig:domain_darcy_in_stokes}]
    {\includegraphics[width=0.34\textwidth,trim={0.cm 0.cm 13.4cm 0.cm},clip]{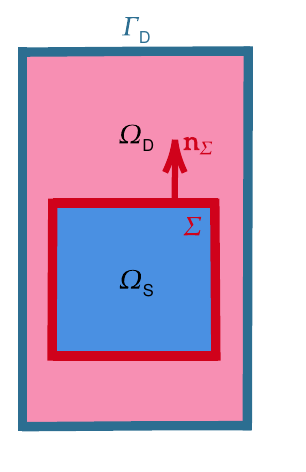}}
    
    \caption{Schematic representations of the two nested Stokes--Darcy domain configurations.}
    \label{fig:DarcyInsideStokes}
\end{figure}

\begin{remark}\label{rem:BC3}
    In addition to the domain setting discussed in Remark~\ref{rem:BCAlt}, we recall that two additional configurations can be taken into account:

    \begin{enumerate}
    \item A Darcy domain immersed in a Stokes domain. In this case, the simply connected assumption of the Stokes domain is not satisfied anymore (see Fig.~\ref{fig:domain_darcy_in_stokes}). Consequently, the stream function is not uniquely determined by the standard boundary conditions. Instead, $\chi=C_i\in \bbR$ on $\Gamma_{i}\subset\Sigma$, where $\Gamma_i$ is the $i$-th connected component of $\Sigma$.     As established in \cite{gr1986}, the following compatibility condition must be satisfied at the continuous level  
    \begin{equation}\label{eq:conditionHoles}
         \oint_{\Sigma} \frac{\partial(\Delta \psi)}{\partial n} \, \mathrm{d}\tau = 0.
     \end{equation}
    Such requirement can be used to uniquely determine the values of $C_i$ (see \cite{Lippke1991} for more details). The additional difficulties associated with the non-simply connected Stokes stream-function formulation do not affect the Darcy domain, and Lemma~\ref{lem:ker} remains valid.

    \item A Stokes domain immersed in a Darcy domain. Alternatively, the inverse configuration is depicted in Fig.~\ref{fig:domain_stokes_in_darcy}.  Since no boundary conditions are defined for the stream functions, we require a different functional setting in order to recover the coercivity (cf. \eqref{coer:a}) of the operator $\bA$. The space related to the stream function is defined by 
    \begin{align}\label{eq:zeroMeanStreamSpace}
        \rH^2_\star(\Omega_\mathrm{S}):= \left\{\xi \in \rH^2(\Omega_\mathrm{S})\setminus \bbR: \int_{\Omega_{\mathrm{S}}}\xi = 0, \int_{\Omega_{\mathrm{S}}}\nabla \xi = \bzero\right\},
    \end{align}
    which together with $\alpha>0$ and the Poincar\'e inequality establish the required coercivity.
    \end{enumerate}
\end{remark}

Observe that Lemma~\ref{lem:ker} naturally motivates the introduction of the Hilbertian space
$$\rH^1_\star(\Omega_\mathrm{D}):=\left\{\psi\in\rH^1(\Omega_\mathrm{D}): \int_{\Omega_\mathrm{D}}\psi = 0\right\},$$
which ensures the uniqueness of the Darcy pressure $\varphi$ in the kernel of the operator on the left-hand side of equation \eqref{eq:operator_form}, in particular, it enforces $\lambda=0$ in Lemma~\ref{lem:ker}. For this reason, in what follows we seek $\varphi\in\rH^1_\star(\Omega_\mathrm{D})$. Notice that this additional constraint allows us to deduce the equivalence between $\|\bullet\|_{1,\Omega_\mathrm{D}}$ and $|\bullet|_{1,\Omega_\mathrm{D}}$ on $\rH^1_\star(\Omega_\mathrm{D})$owing to the Poincar\'e inequality. Consequently, this property leads to the coercivity of the operator $C$, that is,
\begin{alignat}{2}
        \langle \bC\psi,\psi\rangle & \gtrsim  \kappa \|\psi\|^2_{1,\Omega_\mathrm{D}}, &&\quad \forall \psi \in \rH^1_\star(\Omega_\mathrm{D}). \label{coer:c}
\end{alignat}

\begin{remark}
    In light of Remark~\ref{rem:BCAlt}, we readily see that the natural Hilbert space associated to the Darcy pressure $\varphi$ (in both the full Dirichlet and mixed boundary conditions settings) is given by
    $$\rH^1_{\Gamma_\mathrm{D}}(\Omega_\mathrm{D}) := \{\psi\in \rH^1(\Omega_\mathrm{D}): \psi=0 \text{ on } \Gamma_\mathrm{D}\}.$$
    This observation yields an analogue of the coercivity result for the operator $\bC$ (cf. \eqref{coer:c}) in this space. On the other hand, following Remark \ref{rem:BC3}, we note that   \eqref{eq:BCAltS} calls for a different functional setting for the stream function  to  recover the coercivity of the operator $\bA$ (see \eqref{coer:a}). The appropriate choice is the space defined in \eqref{eq:zeroMeanStreamSpace}.
\end{remark}

Let us define the global operator $\cA:\left(\rH^2_{\Gamma_\mathrm{S}}(\Omega_\mathrm{S})\times \rH^1_\star(\Omega_\mathrm{D})\right)\times \left(\rH^2_{\Gamma_\mathrm{S}}(\Omega_\mathrm{S})\times \rH^1_\star(\Omega_\mathrm{D})\right)\to \bbR$ %and $\cF:\rH^2_{\Gamma_\mathrm{S}}(\Omega_\mathrm{S})\times \rH^1_\star(\Omega_\mathrm{D}) \to \bbR$ 
as
\begin{align*}
    \cA((\chi,\varphi),(\xi,\psi))&:=a(\chi,\xi)+b(\xi,\varphi)+b(\chi,\psi)-c(\varphi,\psi),%\\
    %\cF((\xi,\psi))&:= F(\xi) + G(\psi),
\end{align*}
together with the graph norm in the space $\rH^2_{\Gamma_\mathrm{S}}(\Omega_\mathrm{S})\times \rH^1_\star(\Omega_\mathrm{D})$ defined by $\|(\xi,\psi)\|^2:=\|\xi\|^2_{2,\Omega_\mathrm{S}}+\|\psi\|^2_{1,\Omega_\mathrm{D}}$.
We readily see that the boundedness of $\cA$ %and $\cF$ are
is provided by Lemma~\ref{lem:properties_linear_ops}. The global inf-sup condition for the operator $\cA$  is provided below.

\begin{lemma}[global inf-sup condition]\label{lem:global_inf-sup}
    The following bound holds:
    \begin{align*}
        \sup_{(\xi,\psi)\in\left(\rH^2_{\Gamma_\mathrm{S}}(\Omega_\mathrm{S})\times \rH^1_\star(\Omega_\mathrm{D})\right)\setminus \{(0,0)\}} \frac{\cA((\chi,\varphi),(\xi,\psi))}{\|(\xi,\psi)\|} \gtrsim \|(\chi,\varphi)\|, \quad \forall (\chi,\varphi)\in \rH^2_{\Gamma_\mathrm{S}}((\Omega_\mathrm{S})\times \rH^1_\star(\Omega_\mathrm{D}).
    \end{align*}
\end{lemma}
\begin{proof}
    The proof follows by testing with $(\xi,\psi)=(\chi,-\varphi)$. Indeed, using the coercivity of the operators
    $\bA$ and $\bC$ (cf. \eqref{coer:a} and \eqref{coer:c}), we obtain the desired estimate. The details are standard and therefore omitted.
    \qed
\end{proof}
Finally, the well-posedness of the perturbed saddle--point problem and the continuous dependence on data follows from the global inf-sup above combined with a direct application of the Babu\v{s}ka--Brezzi theory.
\begin{theorem}[continuous well-posedness]\label{th:WP}
    There exists a unique solution $(\chi,\varphi)\in\rH^2_{\Gamma_\mathrm{S}}(\Omega_\mathrm{S})\times \rH^1_\star(\Omega_\mathrm{D})$ to \eqref{eq:operator_form} (equivalently, \eqref{eq:WSP}) such that 
    $$\|(\chi,\varphi)\|\lesssim\|\fb\|_{0,\Omega_\mathrm{S}}+\|g\|_{0,\Omega_\mathrm{D}}.$$
\end{theorem}

\begin{comment}
    Now, we provide the inf--sup condition associated to the operator $\bB$ in the following lemma.
    \begin{lemma}
        There exist $\beta>0$ such that
        \begin{alignat*}{2} 
            \sup_{\xi\in\rH^2_{\Gamma_{\mathrm{D}}^{\mathrm{S}}}(\Omega_\mathrm{S})\setminus\{0\}} \frac{\langle \bB \xi,\psi\rangle}{\|\xi\|_{2,\Omega_\mathrm{S}}} &\geq \beta \|\psi\|_{1,\Omega_\mathrm{D}}, &&\quad \forall \xi\in\rH^1_\star(\Omega_\mathrm{D}).
        \end{alignat*}
    \end{lemma}
    \begin{proof}
        \aer{[TODO]}
    \end{proof}
\end{comment}

%%%%%%%%%%%%%%%%%%%%%%%%%%%%%%%%%%%%%%%%%%%%%%%%%
\section{Virtual element discretisation}\label{sec:vem}
This section introduces the conforming lowest-order virtual element discretisation, i.e., quadratic $C^1$--\ac{vem} for stream functions $\chi$ and linear $C^0$--\ac{vem} for pressure $\varphi$. In addition, we define the associated polynomial projection operators and provide the corresponding approximation and interpolation estimates. Finally, we introduce the discrete version of \eqref{eq:WSP}.

\paragraph{Admissible meshes and scaled monomials.} Let $\cT_h$ be a collection of polygonal meshes on $\Omega_\mathrm{S}\cup\Omega_\mathrm{D}$ and $\cE_h$ be the set of all edges. The diameter of a polygon $K\in\cT_h$ (resp. edge $e\in\cE_h$) is denoted by $h_K$ (resp. $h_e$). The maximum diameter of elements in $\cT_h$ is represented by $h$. It is assumed that there exists a constant $\eta>0$ such that
\begin{enumerate}[label={\textbf{(M\arabic*)}}, align=left, leftmargin=*, labelwidth=!, labelsep=1em]
    \item \label{A1} every polygon $K$ is star-shaped with respect to a ball of a radius greater than $\eta h_K$, 
    \item \label{A2} every edge $e\in \partial K$ satisfies the inequality $h_e \geq \eta h_K$. 
\end{enumerate}
In view of the boundary conditions,  we denote by $\bt_e$ and $\bn_e$ the unit tangential and normal vector of $e$ with respect to $\partial K$; we simply write $\bt$ and $\bn$ when the context is sufficiently clear. The polygonal mesh associated to $\Omega_\mathrm{S}$ (resp. $\Omega_\mathrm{D}$) is defined by $\cT_h(\Omega_\mathrm{S}):=\{K\in \cT_h: K\subset \Omega_\mathrm{S}\}$ (resp. $\cT_h(\Omega_\mathrm{D}):=\{K\in \cT_h: K\subset \Omega_\mathrm{D}\}$), while the set of edges lying in the interface $\Sigma$ is given by $\cE_h(\Sigma):=\{e\in\cE_h:e\subset\Sigma\}$ and $\cT_h(\Omega_\mathrm{S})\cap\cT_h(\Omega_\mathrm{D})=\cE_h(\Sigma)$.

Let $K\in\cT_h$ be a polygon with barycenter $\bx_K :=(x_{1,K},x_{2,K})^{\tt t}$, and $\rM_k(K)$ for $k\geq0$ be the set of scalar scaled monomials  defined by
\begin{align*}
    \rM_k(K):=\left\{ \left( \frac{\bx-\bx_P}{h_P} \right)^{\balpha}, 0\leq |\balpha|\leq k \right\},
\end{align*}
where $\boldsymbol{\alpha}=(\alpha_1,\alpha_2)^{\tt t}$ is a non-negative multi-index such that $\bx^{\balpha}:=x_1^{\alpha_1} x_2^{\alpha_2}$ with $\bx=(x_1,x_2)^{\tt t}$ and $|\boldsymbol{\alpha}|=\alpha_1+\alpha_2$. %\rek{Similarly, the vector version is defined by} $\bM_k:=\{(m,0):m\in\rM_k(K)\}\cup\{(0,m):m\in\rM_k(K)\}$.
Note that $\rM_k(K)$ %(resp. $\bM_k(K)$) 
defines a basis for the space of scalar %(resp. vector) 
polynomials of degree at most $k$, denoted by $\rP_k(K)$. %(resp. $\bP_k(K)$).

\subsection{Polynomial projection operators}\label{sec:proj_inter_op} For each polygon $K\in \cT_h$, we introduce the following local polynomial projection operators:
\begin{enumerate}
    \item  The $\rL^2$-projection $\Pi_1^{0,K}: \rL^2(K)\to\rP_1(K)$ is defined, for any $\phi\in \rL^2(K)$, by
    \begin{align}\label{L2_proj}
        \int_K (\Pi_1^{0,K}\phi-\phi) m_1 = 0,\quad \forall m_1\in\rM_1(K).
    \end{align}
    % Moreover, the vectorial counterpart is denoted by $\bPi_1^{0,K}$, and is defined analogously.
    \item  The $\rH^1$-projection $\Pi_1^{\nabla,K}: \rH^1(K)\to\rP_1(K)$ is defined, for any $\phi\in \rH^1(K)$, by
    \begin{subequations}\label{H1_proj}
        \begin{align}
            \int_K \nabla(\Pi_1^{\nabla,K}\phi-\phi)\cdot \nabla m_1 &= 0, \quad \forall m_1\in\rM_1(K),\\
            \int_{\partial K} \Pi_1^{\nabla,K}\phi-\phi &=  0.
        \end{align}
    \end{subequations}
    \item The $\rH^2$-projection $\Pi_2^{\nabla^2,K}: \rH^2(K)\to\rP_2(K)$ is defined, for any $\phi\in \rH^2(K)$, by
    \begin{subequations}\label{H2_proj}
        \begin{align}
            \int_K \nabla^2(\Pi_2^{\nabla^2,K}\phi-\phi) : \nabla^2 m_2 &= 0, \quad \forall m_2\in\rM_2(K),\\
            \int_{\partial K}\Pi_2^{\nabla^2,K}\phi - \phi &= 0, \\
            \int_{\partial K} \nabla(\Pi_2^{\nabla^2,K}\phi - \phi) &= 0.
        \end{align}
    \end{subequations}
\end{enumerate}

We finalise this section by recalling a classical result from polynomial approximation theory \cite{brenner08}. 

\begin{lemma}[polynomial projection estimates] \label{lem:est_poly}
    For any $K\in \cT_h$, suppose that $\phi\in  \rH^{s}(K)$ and $\phi_\pi\in \rP_{k+1}(K)$, with  $0 \leq s \leq k+1$. Then, there exists a positive constant that depends only on $\eta$ ({cf.} \textup{\ref{A1}-\ref{A2}}) such that, for $0\leq r\leq s$ the following estimate holds
    \begin{align*}
        |\phi - \phi_\pi|_{r,K} &\lesssim  h_K^{s-r}|\phi|_{s,K}.
    \end{align*} 
\end{lemma}
\begin{remark}
    In practice, the polynomial approximation $\phi_\pi$ of $\phi\in \rH^s(K)$ corresponds to the polynomial projections $\Pi_{1}^{0,K}\phi$, $\Pi_{1}^{\nabla,K}\phi$, or $\Pi_{2}^{\nabla^2,K}\phi$ for $s\ge 0, 1,$ or $2$, respectively. 
\end{remark}

\subsection{\texorpdfstring{$C^1$}{C1}--\texorpdfstring{\ac{vem}}{VEM} space}\label{sec:C1VEM} 
To approximate the functional space associated to the stream functions we consider the following local enhanced virtual space for any $K\in \cT_h(\Omega_\mathrm{S})$ given by
\begin{align*}
    \Xi_h(K)
    :=\biggl\{\xi_h\in \rH^2(K) : &\, \Delta^2 \xi_h\in\rP_2(K),\\
    &\, \xi_h|_{\partial K}\in C^0(\partial K),\,
    \xi_h|_e\in\rP_3(e),\, \forall e\in\partial K,\\
    &\, \nabla \xi_h|_{\partial K}\in [C^0(\partial K)]^2,\,
    (\nabla\xi_h\cdot \bn_e)|_e\in\rP_1(e),\, \forall e\in\partial K,\\
    &\, \int_{K} (\Pi_2^{\nabla^2} \xi_h-\xi_h) m_2=0,\, \forall m_2\in\rM_{2}(K)\biggr\}.
\end{align*}
We select the unisolvent set of \acp{dof} for $\Xi_h(K)$ as follows:
\begin{enumerate}[
    label={($\mathbf{D\arabic*}_{\xi_h}$)}, 
    leftmargin=\leftmargini + 2.25em, 
    align=left,               
    labelsep=0.5em            
]
    \item the values of $\xi_h$ at the vertices of $K$,
    \item the values of $h_{V_i}\nabla \xi_h$ at the vertices of $K$,
\end{enumerate}
where $h_{V_i}$ corresponds to the average of diameters $h_K$ corresponding to the elements $K$ that have $V_i$ as a vertex. Thanks to the enhancement condition, the projection operators $\Pi_{2}^{\nabla^2,K}$ and $\Pi_{2}^{0,K}$ coincide (see \cite{MRV2018}). Furthermore, they can be directly computed from the \acp{dof} provided above. We refer to \cite{absv2016} for further details regarding unisolvence and computability. The global virtual space for stream functions is defined as follows
\begin{align*}
  \Xi_h :=\left\{\xi_h\in\rH^2_{\Gamma_\mathrm{S}}(\Omega_\mathrm{S}): \xi_h|_K\in \Xi_h(K),\, \forall K\in\cT_h(\Omega_\mathrm{S})\right\}.
\end{align*}
Note that the space $\Xi_h$ has $3N_\mathrm{S}$ \acp{dof}, where $N_\mathrm{S}$ is the total number of interior vertices of the polygonal mesh $\cT_h(\Omega_\mathrm{S})$. Moreover, notice that the patching of local spaces across the mesh skeleton ensures global $C^1$-conformity, as the shared degrees of freedom uniquely determine both the function and its normal derivative along element interfaces.

\subsection{\texorpdfstring{$C^0$}{C0}--\texorpdfstring{\ac{vem}}{VEM} space}\label{sec:C0VEM}
 To approximate the Darcy pressure space, for any $K\in \cT_h(\Omega_\mathrm{D})$ we consider the local enhanced virtual space given by
 \begin{align*}
     \Psi_h(K) := \bigg\{ \psi_h \in \rH^1(K): &\, \Delta \psi_h \in \rP_1(K),\\
     &\,\psi_h|_e\in C^0(\partial K),\, \psi_h|_e\in\rP_1(e),\, \forall e\in\partial K,\\
     &\int_K(\Pi_{1}^{\nabla,K}\psi_h-\psi_h)m_1=0,\,\forall m_1\in\rM_1(K)\bigg\},
 \end{align*}
equipped with the following unisolvent set of \acp{dof}:
\begin{enumerate}[
    label={($\mathbf{D\arabic*}_{\psi_h}$)}, 
    leftmargin=\leftmargini + 2.25em, 
    align=left,               
    labelsep=0.5em            
]
    \item The values of $\psi_h$ at the vertices of $K$.
\end{enumerate}
We recall that the projection operators $\Pi_{1}^{\nabla,K}$ and $\Pi_{1}^{0,K}$ coincide on $\Psi_h(K)$ and are computable via the previously defined \acp{dof}. Additionally, the unisolvence of the space is established in \cite{bbcmmr13,bbmr13}. Similarly, the global \ac{vem} pressure space is defined by 
\begin{align*}
    \Psi_h:=\left\{\psi_h\in \rH^1_\star(\Omega_\mathrm{D}): \psi_h|_K\in \Psi_h(K),\, \forall K\in\cT_h(\Omega_\mathrm{D})\right\}.
\end{align*}
For this space, there are $N_\mathrm{D}$ total number of \acp{dof} for the space $\Psi_h$, with $N_\mathrm{D}$ denoting the number of interior vertices of the polygonal mesh 
$\cT_h(\Omega_\mathrm{D})$. In addition, the global space $\Xi_h$ is constructed by gluing the local spaces $\Xi_h(K)$ across the mesh skeleton, where the matching of vertex values guarantees $C^0$-continuity across element interfaces. Finally, at the implementation level, the zero mean condition of the space $\rH^1_\star(\Omega_\mathrm{D})$ is imposed through a scalar Lagrange multiplier.

\subsection{Discrete formulation and unique solvability}\label{sec:disc_formulation}
First, we define the discrete bilinear forms $a_h(\bullet,\bullet):\Xi_h\times \Xi_h\to \bbR$, $b(\bullet,\bullet):\Xi_h\times \Psi_h\to \bbR$, $c_h(\bullet,\bullet):\Psi_h\times \Psi_h\to \bbR$, and the linear forms $\rF_h(\bullet):\Xi_h\to \bbR$, $\rG_h(\bullet):\Psi_h\to \bbR$ locally as
\begin{align*}
    a_h(\chi_h,\xi_h) &= \langle \bA_h \chi_h,\xi_h \rangle = \sum_{K\in\cT_h(\Omega_\mathrm{S})} a_h^K(\chi_h,\xi_h) %= \sum_{K\in\cT_h(\Omega_\mathrm{S})} \langle \bA_h^K \chi_h,\xi_h \rangle\\
    = \sum_{K\in\cT_h(\Omega_{\mathrm{S}})} \left[ a^{\nabla^2,K}_h(\chi_h,\xi_h) + a^{\Sigma,K}(\chi_h,\xi_h) \right]\\
    &:= \sum_{K\in\cT_h(\Omega_\mathrm{S})} \bigg[\mu\int_{K} \nabla^2(\Pi_{2}^{\nabla^2,K}\chi_h) : (\nabla^2\Pi_{2}^{\nabla^2,K}\xi_h) + S^K_\mathrm{S}(\chi_h-\Pi_{2}^{\nabla^2,K}\chi_h,\xi_h-\Pi_{2}^{\nabla^2,K}\xi_h)\\
    & \quad + \sum_{e\in \cE_h(\Sigma)\cap \partial K} \frac{\alpha\mu}{\sqrt{\kappa}}\int_e (\nabla \chi_h\cdot\bn_\Sigma) (\nabla \xi_h\cdot\bn_\Sigma)\bigg],\\
    b(\xi_h,\psi_h) &= \langle \bB \xi_h, \psi_h \rangle = \sum_{K\in\cT_h} b^{K}(\xi_h,\psi_h) \\ %= \sum_{K\in\cT_h} \langle \bB^{K} \xi_h, \psi_h \rangle \\
    &:= \sum_{K\in\cT_h}\sum_{e\in\cE_h(\Sigma)\cap \partial K}\frac{1}{2}\int_e \psi_h \nabla \xi_h\cdot\bt_\Sigma = \sum_{e\in\cE_h(\Sigma)}\int_e \psi_h \nabla \xi_h\cdot\bt_\Sigma,\\
    c_h(\varphi_h,\psi_h) &= \langle \bC_h \varphi_h, \psi_h \rangle = \sum_{K\in\cT_h(\Omega_\mathrm{D})}c_h^K(\varphi_h,\psi_h) \\%= \sum_{K\in\cT_h(\Omega_\mathrm{D})}\langle \bC_h^K \varphi_h, \psi_h \rangle\\
    &:= \sum_{K\in\cT_h(\Omega_\mathrm{D})}\kappa\int_{K} \nabla (\Pi_{1}^{\nabla,K}\varphi_h) \cdot \nabla (\Pi_{1}^{\nabla,K}\psi_h) + S^K_\mathrm{D}(\varphi_h-\Pi_{1}^{\nabla,K}\varphi_h,\psi_h-\Pi_{1}^{\nabla,K}\psi_h),\\
    \rF_h(\xi_h) &= \sum_{K\in\cT_h(\Omega_\mathrm{S})} \rF_h^K(\xi_h) := \sum_{K\in\cT_h(\Omega_\mathrm{S})} \int_{K} \fb \cdot \bcurl (\Pi_{2}^{\nabla^2,K}\xi_h),\\
    \rG_h(\psi_h) &= \sum_{K\in\cT_h(\Omega_\mathrm{D})} \rG_h^K(\psi_h) := -\sum_{K\in\cT_h(\Omega_\mathrm{D})} \int_{K}g\Pi_{1}^{0,K}\psi_h,
\end{align*}
where $\bA_h:\Xi_h\to(\Xi_h)'$ (resp. $\bC_h:\Psi_h \to (\Psi_h)'$) is the linear operator induced by $a_h(\bullet,\bullet)$ (resp. $c_h(\bullet,\bullet)$), the superscript $K$ indicates the restriction to the polygon $K$, and the stabilisation operators $S^K_\mathrm{S}:\Xi_h(K)\times\Xi_h(K) \to \bbR$ and $S^K_\mathrm{D}:\Psi_h(K)\times\Psi_h(K) \to \bbR$ are any positive semi-definite inner products satisfying 
\begin{subequations}\label{eq:stabilisation}
    \begin{alignat}{2}  
        a^{\nabla^2,K}(\xi_h,\xi_h)&\lesssim S^K_\mathrm{S}(\xi_h,\xi_h)\lesssim a^{\nabla^2,K}(\xi_h,\xi_h), &&\qquad \forall \xi_h\in \mathrm{ker}(\Pi_{2}^{\nabla^2,K}),\label{eq:stabilisationS}\\
        c^{K}(\psi_h,\psi_h)&\lesssim S^K_\mathrm{D}(\psi_h,\psi_h)\lesssim c^{K}(\psi_h,\psi_h), &&\qquad \forall \psi_h\in \mathrm{ker}(\Pi_{1}^{\nabla,K}).\label{eq:stabilisationD}
    \end{alignat}
\end{subequations}

Next, we introduce the discrete version of the stream function -- pressure weak formulation for the Stokes--Darcy interface problem (cf. \eqref{eq:WSP}) as follows: 
given $\fb\in \bL^2(\Omega_\mathrm{S})$ and
$g\in\rL^2_0(\Omega_\mathrm{D})$, find $(\chi_h,\varphi_h)\in \Xi_h \times \Psi_h$ such that 
\begin{subequations}\label{eq:WSP_h}
    \begin{alignat}{2}
        a_h(\chi_h,\xi_h) + b(\xi_h,\varphi_h) &= \rF_h(\xi_h), &&\,\,\quad \forall \xi_h\in \Xi_h, \label{eq:WSP_h1} \\
        b(\chi_h,\psi_h) - c_h(\varphi_h,\psi_h) &= \rG_h(\psi_h),  && \,\,\quad \forall \psi_h\in \Psi_h. \label{eq:WSP_h2}
    \end{alignat}
\end{subequations}
Similarly, the respective perturbed saddle-point (in operator form) equivalent to \eqref{eq:WSP_h} reads
\begin{equation}\label{eq:operator_form_h} 
    \begin{pmatrix}
        \bA_h & \bB^* \\
        \bB & -\bC_h 
    \end{pmatrix}
    \begin{pmatrix}
        \chi_h \\ 
        \varphi_h
    \end{pmatrix}
    = 
    \begin{pmatrix}
        \rF_h \\ 
        \rG_h    
    \end{pmatrix} 
    ,\quad \text{in } [\Xi_h\times \Psi_h]'. 
\end{equation}

The properties of the discrete linear operators are provided in the following lemma, recalling that the relation $\|\bullet\|^2_{s,\Omega_\square}=\displaystyle{\sum_{K\in \cT_h(\Omega_\square)}}\|\bullet\|^2_{s,K}$ holds for $s\in\{0,1,2\}$ and $\square\in\{\mathrm{S},\mathrm{D}\}$. 
\begin{lemma}[properties of the discrete operators]\label{lem:properties_discrete_operators}
    The following bounds hold
    \begin{subequations}\label{eq:boundsh}
        \begin{alignat}{2}
        |\langle \bA_h\chi_h,\xi_h \rangle| & \lesssim \mu\max\left\{1,\frac{\alpha}{\sqrt{\kappa}}\right\} \|\chi_h\|_{2,\Omega_\mathrm{S}}\|\xi_h\|_{2,\Omega_\mathrm{S}}, &&\quad \forall  \chi_h,\xi_h \in \Xi_h,\label{bound:ah}\\
        \langle \bA_h\xi_h,\xi_h \rangle & \gtrsim \mu \|\xi_h\|^2_{2,\Omega_\mathrm{S}}, &&\quad \forall  \xi_h \in \Xi_h,\label{coer:ah}\\
        %|\langle \bB_h \xi_h,\psi_h\rangle| & \lesssim \|\xi_h\|_{2,\Omega_\mathrm{S}}\|\psi_h\|_{1,\Omega_\mathrm{D}}, &&\quad \forall \xi_h \in \Xi_h,\, \forall\psi_h \in \Psi_h,\label{bound:bh}\\
        |\langle \bC_h\varphi_h,\psi_h\rangle| & \lesssim \kappa\|\varphi_h\|_{1,\Omega_\mathrm{D}}\|\psi_h\|_{1,\Omega_\mathrm{D}}, &&\quad  \forall \varphi_h,\psi_h \in \Psi_h, \label{bound:ch}\\
        \langle \bC_h\psi_h,\psi_h\rangle & \gtrsim  \kappa \|\psi_h\|^2_{1,\Omega_\mathrm{D}}, &&\quad \forall \psi_h \in \Psi_h, \label{coer:ch}\\
        |\rF_h(\xi_h)| & \lesssim \|\fb\|_{0,\Omega_\mathrm{S}}\|\xi_h\|_{2,\Omega_\mathrm{S}}, &&\quad \forall \xi_h \in \Xi_h, \label{bound:Fh}\\
        |\rG_h(\psi_h)| & \lesssim \|g\|_{0,\Omega_\mathrm{D}}\|\psi_h\|_{1,\Omega_\mathrm{D}}, &&\quad \forall \psi_h \in \Psi_h.\label{bound:Gh}
        \end{alignat}
    \end{subequations}
\end{lemma}
\begin{proof}
    Standard arguments (cf. Lemma~\ref{lem:properties_linear_ops}) together with the stabilisation bounds \eqref{eq:stabilisation} lead to the result, further details are omitted.
    \qed
\end{proof}

We now define the discrete global operator $\cA_h:(\Xi_h\times \Psi_h)\times (\Xi_h\times \Psi_h)\to \bbR$ %and $\cF:\rH^2_{\Gamma_\mathrm{S}}(\Omega_\mathrm{S})\times \rH^1_\star(\Omega_\mathrm{D}) \to \bbR$ 
as
\begin{align*}
    \cA_h((\chi_h,\varphi_h),(\xi_h,\psi_h))&:=a_h(\chi_h,\xi_h)+b(\xi_h,\varphi_h)+b(\chi_h,\psi_h)-c_h(\varphi_h,\psi_h),%\\
    %\cF((\xi,\psi))&:= F(\xi) + G(\psi),
\end{align*}
and the graph norm in the space $\Xi_h\times\Psi_h$ is defined by $\|(\xi_h,\psi_h)\|^2:=\|\xi_h\|^2_{2,\Omega_\mathrm{S}}+\|\psi_h\|^2_{1,\Omega_\mathrm{D}}$.
Similarly, the boundedness of $\cA_h$ %and $\cF$ are
is provided by Lemma~\ref{lem:properties_discrete_operators}. Now, the discrete global inf-sup condition for the operator $\cA_h$  is provided below as a consequence of \eqref{coer:ah} and \eqref{coer:ch} (see also Theorem~\ref{lem:global_inf-sup}).
\begin{lemma}[discrete global inf-sup]\label{discrete_global_inf-sup}
    The following bound holds:
    \begin{align*}
        \sup_{(\xi_h,\psi_h)\in(\Xi_h\times\Psi_h)\setminus \{(0,0)\}} \frac{\cA_h((\chi_h,\varphi_h),(\xi_h,\psi_h))}{\|(\xi_h,\psi_h)\|} \gtrsim \|(\chi_h,\varphi_h)\|, \quad \forall (\chi_h,\varphi_h)\in \Xi_h\times\Psi_h.
    \end{align*}
\end{lemma}
The well-posedness of the discrete perturbed saddle--point problem (cf. \eqref{eq:operator_form_h}) and the discrete continuous dependence on data is stated below. %as a consequence of the discrete global inf-sup above combined with a direct application of Brezzi theory.
\begin{theorem}[discrete well-posedness]\label{th:WPh}
    There exists a unique solution $(\chi_h,\varphi_h)\in\Xi_h\times\Psi_h$ to \eqref{eq:operator_form_h} (equivalently, \eqref{eq:WSP_h}) such that 
    $$\|(\chi_h,\varphi_h)\|\lesssim\|\fb\|_{0,\Omega_\mathrm{S}}+\|g\|_{0,\Omega_\mathrm{D}}.$$
\end{theorem}
Note that the hidden constant in Theorem~\ref{th:WPh} may differ from that in  Theorem~\ref{th:WP}.

\section{A priori error analysis}\label{sec:apriori}
This section is devoted to the derivation of an a priori error estimate for the \ac{vem} discretisation discussed in Section~\ref{sec:vem}. We start by recalling the interpolation properties of the spaces $\Xi_h$ and $\Psi_h$ (see e.g. \cite{absv2016,cangiani2017posteriori,kmr25,MRR15,bbcmmr13}).
\begin{lemma}[interpolation estimates]
\label{lem:est_inte}
Let $\xi\in\rH^2_{\Gamma_\mathrm{S}}(\Omega_\mathrm{S})\cap\rH^{r_1}(\Omega_\mathrm{S})$ and $\psi\in\rH^1_\star(\Omega_\mathrm{D})\cap\rH^{r_2}(\Omega_\mathrm{D})$ with $2\leq r_1 \leq 3$ and $1\leq r_2 \leq 2$. Then, there exist interpolation operators $\rI_{\mathrm{S},h}:\rH^2_{\Gamma_\mathrm{S}}(\Omega_\mathrm{S})\cap\rH^{r_1}(\Omega_\mathrm{S})\rightarrow \Xi_h$ and $\rI_{\mathrm{D},h}:\rH^1_\star(\Omega_\mathrm{D})\cap\rH^{r_2}(\Omega_\mathrm{D})\rightarrow \Psi_h$, such that
\begin{alignat*}{2}
    |\xi-\rI_{\mathrm{S},h}^{K_{\mathrm{S}}}\xi|_{j_1,K_{\mathrm{S}}} &\lesssim %C_{\mathrm{I}_{2}}
    h^{r_1-j_1}_K |\xi|_{r_1,K_{\mathrm{S}}}, &&\quad 0\leq j_1 \leq 2, \quad \forall K_{\mathrm{S}}\in\cT_h(\Omega_\mathrm{S}), \\
    |\psi-\rI_{\mathrm{D},h}^{K_{\mathrm{D}}}\psi|_{j_2,K_{\mathrm{D}}} &\lesssim %C_{\mathrm{I}_{1}}
    h^{r_2-j_2}_K |\psi|_{r_2,K_{\mathrm{D}}}, &&\quad 0\leq j_2 \leq 1, \quad \forall K_{\mathrm{D}}\in\cT_h(\Omega_\mathrm{D}),
\end{alignat*}
where $\rI_{\square,h}^{K_{\square}}$ denotes the restriction of the global interpolant $\rI_{\square,h}$ to the local element $K_{\square}$ for both the solid and fluid subdomains, i.e., $\square \in\{\mathrm{S},\mathrm{D}\}$.
\end{lemma}

We now define the total error as
\begin{align*}
\text{e}_h&:=\|(\chi-\chi_h,\varphi-\varphi_h)\|=\left(\text{e}_{\chi_h}^2+\text{e}_{\varphi_h}^2\right)^{\frac{1}{2}},
\end{align*}
where $\text{e}_{\chi_h}:=\|\chi-\chi_h\|_{2,\Omega_{\mathrm{S}}}$ and $\text{e}_{\varphi_h}:=\|\varphi-\varphi_h\|_{1,\Omega_{\mathrm{D}}}$. The following result provides a bound of $\text{e}_h$ in terms of the data approximation, polynomial approximation, and interpolation errors.
\begin{lemma}[energy-error estimate]\label{lem:errorForApriori}
    Let $(\chi,\varphi)\in\rH^2_{\Gamma_\mathrm{S}}(\Omega_\mathrm{S})\times \rH^1_\star(\Omega_\mathrm{D})$ and $(\chi_h,\varphi_h)\in\Xi_h\times\Psi_h$ be the unique solutions to \eqref{eq:operator_form} and \eqref{eq:operator_form_h}, respectively. Then, the following estimate holds:
    \begin{align*}
    \text{e}_h &\lesssim \max\left\{ 1,\mu,\kappa, \frac{\mu\alpha}{\sqrt{\kappa}} \right\} \bigg( \sum_{K\in \cT_h(\Omega_\mathrm{S})} \bigg[ |\chi-\rI_{\mathrm{S},h}^K\chi|_{2,K} + |\chi-\Pi_2^{\nabla^2,K}\chi|_{2,K} + \|F^K - F_h^K\|_{(\Xi_h(K))'} \bigg] \\
    &\quad + \sum_{K\in \cT_h(\Omega_\mathrm{D})} \bigg[ |\varphi-\rI_{\mathrm{D},h}^K\varphi|_{1,K} + |\varphi-\Pi_1^{\nabla,K}\varphi|_{1,K} + \|G^K - G_h^K\|_{(\Psi_h(K))'} \bigg] \bigg).
    \end{align*}
\end{lemma}
\begin{proof}
    From the unique solvability of both continuous weak formulation \eqref{eq:WSP} and the discrete weak formulation \eqref{eq:WSP_h} we readily see that $(\chi_h-\rI_{\mathrm{S},h}\chi,\varphi_h-\rI_{\mathrm{D},h}\varphi)\in\Xi_h\times \Psi_h$ is the unique solution to
    \begin{alignat*}{2}
        a_h(\chi_h-\rI_{\mathrm{S},h}\chi,\xi_h) + b(\xi_h,\varphi_h-\rI_{\mathrm{D},h}\varphi) &= \check{\rF}_h(\xi_h), &&\,\, \forall \xi_h\in \Xi_h, \\
        b(\chi_h-\rI_{\mathrm{S},h}\chi,\psi_h) - c_h(\varphi_h-\rI_{\mathrm{D},h}\varphi,\psi_h) &= \check{\rG}_h(\psi_h),  && \,\, \forall \psi_h\in \Psi_h,
    \end{alignat*}
    where the right-hand sides are given by
    \begin{subequations}\label{auxiliar_functionals1}
        \begin{align}
            \check{\rF}_h(\xi_h) &:= \left[ a(\chi,\xi_h) - a_h(\rI_{\mathrm{S},h}\chi,\xi_h) \right] + b(\xi_h,\varphi-\rI_{\mathrm{D},h}\varphi) + \left[ F_h - F \right](\xi_h),\label{auxiliar_functional1}\\
            \check{\rG}_h(\psi_h) &:= \left[ c_h(\rI_{\mathrm{D},h}\varphi,\psi_h) - c(\varphi,\psi_h) \right] + b(\chi-\rI_{\mathrm{S},h}\chi,\psi_h) + \left[ G_h - G \right](\psi_h).\label{auxiliar_functional2}
        \end{align}
    \end{subequations}
    The continuous dependence on data for the discrete problem (cf. Theorem~\ref{th:WPh}) implies that
    $$\|(\chi_h-\rI_{\mathrm{S},h}\chi,\varphi_h-\rI_{\mathrm{D},h}\varphi)\|\lesssim \|\check{\rF}_h\|_{\Xi_h'} + \|\check{\rG}_h\|_{\Psi_h'}.$$
    
    We now focus in bounding the functionals appearing in \eqref{auxiliar_functionals1}. First, we recall the polynomial consistency of the discrete bilinear forms $a^{\nabla^2}_h(\bullet,\bullet)$ and $c_h(\bullet,\bullet)$, i.e.,
    \begin{alignat*}{2}
        a^{\nabla^2,K}_h(\Pi_2^{\nabla^2,K}\chi_h,\xi_h) &= a^{\nabla^2,K}(\Pi_2^{\nabla^2,K}\chi_h,\xi_h), &&\quad \forall \xi_h\in\Xi_h, \\
        c_h^K(\Pi^{\nabla,K}_1\varphi_h,\psi_h) &= c^K(\Pi^{\nabla,K}_1\varphi_h,\psi_h), &&\quad \forall \psi_h\in\Psi_h.
    \end{alignat*}
    Thus, \eqref{auxiliar_functionals1} turns into
    \begin{align*}
        \check{\rF}_h(\xi_h) &= \sum_{K\in\cT_h(\Omega_{\mathrm{S}})}\left[ a^{\nabla^2,K}(\chi-\Pi_2^{\nabla^2,K}\chi_h,\xi_h) - a^{\nabla^2,K}_h(\rI_{\mathrm{S},h}^K\chi-\Pi_2^{\nabla^2,K}\chi_h,\xi_h) \right] \notag \\ 
        &\quad + \sum_{K\in\cT_h(\Omega_{\mathrm{S}})}a^{\Sigma,K}(\chi-\rI_{\mathrm{S},h}^K\chi,\xi_h) + \sum_{K\in\cT_h} b^K(\xi_h,\varphi-\rI_{\mathrm{D},h}^K\varphi) + \sum_{K\in\cT_h(\Omega_{\mathrm{S}})}\left[ F_h^K - F^K \right](\xi_h),%\label{auxiliar_functional3}
        \\
        \check{\rG}_h(\psi_h) &= \sum_{K\in\cT_h(\Omega_{\mathrm{D}})}\left[ c_h^K(\rI_{\mathrm{D},h}^K\varphi-\Pi^{\nabla,K}_1\varphi_h,\psi_h) - c^K(\varphi-\Pi^{\nabla,K}_1\varphi_h,\psi_h) \right] \notag \\
        &\quad + \sum_{K\in\cT_h} b^K(\chi-\rI_{\mathrm{S},h}^K\chi,\psi_h) + \sum_{K\in\cT_h(\Omega_{\mathrm{D}})}\left[ G_h^K - G^K \right](\psi_h).%\label{auxiliar_functional4}
    \end{align*}
    The triangle inequality, the Cauchy--Schwarz inequality, and the stabilisation properties in \eqref{eq:stabilisation} yield
    \begin{align*}
        \left|\check{\rF}_h(\xi_h)\right| &\lesssim \max\left\{1,\mu,\frac{\mu\alpha}{\sqrt{\kappa}}\right\} \bigg( \sum_{K\in\cT_h(\Omega_{\mathrm{S}})}\left[ \left(|\chi-\Pi_2^{\nabla^2,K}\chi_h|_{2,K} 
        + |\rI_{\mathrm{S},h}^K\chi-\Pi_2^{\nabla^2,K}\chi_h|_{2,K} \right) \|\xi_h\|_{2,K} \right] \notag \\
        &\quad + \sum_{e\in\cE_h(\Sigma)} \|\nabla(\chi-\rI_{\mathrm{S},h}^K\chi)\cdot \bn_\Sigma\|_{0,e} \|\nabla\xi_h\cdot \bn_\Sigma\|_{0,e} + \sum_{e\in\cE_h(\Sigma)} \|\nabla\xi_h\cdot\bn_\Sigma\|_{0,e} \|\varphi-\rI_{\mathrm{D},h}\varphi\|_{0,e} \notag \\
        &\quad + \sum_{K\in\cT_h(\Omega_{\mathrm{S}})}\left[ F_h^K - F^K \right](\xi_h)\bigg),\\
        \left|\check{\rG}_h(\psi_h)\right| &\lesssim \max\left\{1,\kappa\right\} \bigg(\sum_{K\in\cT_h(\Omega_{\mathrm{D}})}\left[ \left(|\rI_{\mathrm{D},h}^K\varphi-\Pi^{\nabla,K}_1\varphi_h|_{1,K} + |\varphi-\Pi^{\nabla,K}_1\varphi_h|_{1,K}\right) \|\psi_h\|_{1,K} \right] \notag \\
        &\quad + \sum_{e\in\cE_h(\Sigma)} \|\nabla(\chi-\rI_{\mathrm{S},h}^K\chi)\cdot \bn_\Sigma\|_{0,e} \|\psi_h\|_{0,e}  + \sum_{K\in\cT_h(\Omega_{\mathrm{D}})}\left[ G_h^K - G^K \right](\psi_h) \bigg).
    \end{align*}
    Next, we apply the discrete trace inequality and the inverse estimates in \cite[Lemma 3.5]{HUANG2021} and \cite[Theorem 3.6]{Chen2018} to obtain
    \begin{align*}
        \|\nabla \xi_h \cdot \bn\|_{0,e} \lesssim h_K^{-\frac{1}{2}} |\xi_h|_{1,K} \quad \text{and} \quad \|\psi_h\|_{0,e} \lesssim h_K^{-\frac{1}{2}} \|\psi_h\|_{0,K}. 
    \end{align*}
    This, the triangle inequality and the inequality $\sum_i a_ib_i \leq (\sum_i a_i) (\sum_i b_i)$ (for any finite nonnegative sequences $(a_i)_i$ and $(b_i)_i$) lead to
    \begin{align*}
        \left|\check{\rF}_h(\xi_h)\right| &\lesssim \max\left\{1,\mu,\frac{\mu\alpha}{\sqrt{\kappa}}\right\} \bigg( \Big(  \sum_{K\in\cT_h(\Omega_{\mathrm{S}})}\left[ |\chi-\rI_{\mathrm{S},h}^K\chi|_{2,K} + |\chi-\Pi_2^{\nabla^2,K}\chi_h|_{2,K}\right] 
        \Big) \|\xi_h\|_{2,\Omega_{\mathrm{S}}}  \notag \\
        &\quad + \Big(\sum_{e\in\cE_h(\Sigma)} h_K^{-\frac{1}{2}}\|\nabla(\chi-\rI_{\mathrm{S},h}^K\chi)\cdot \bn_\Sigma\|_{0,e}\Big) \|\xi_h\|_{2,\Omega_{\mathrm{S}}}  \notag \\
        &\quad + \Big(\sum_{e\in\cE_h(\Sigma)} h_K^{-\frac{1}{2}}\|\varphi-\rI_{\mathrm{D},h}^K\varphi\|_{0,e}\Big) \|\xi_h\|_{2,\Omega_{\mathrm{S}}} + \sum_{K\in\cT_h(\Omega_{\mathrm{S}})}\left[ F_h^K - F^K \right](\xi_h)\bigg),\\
        \left|\check{\rG}_h(\psi_h)\right| &\lesssim \max\left\{1,\kappa\right\} \bigg( \Big( \sum_{K\in\cT_h(\Omega_{\mathrm{D}})}\left[ |\varphi-\rI_{\mathrm{D},h}^K\varphi|_{1,K} + |\varphi-\Pi^{\nabla,K}_1\varphi_h|_{1,K}\right]\Big) \|\psi_h\|_{1,\Omega_{\mathrm{D}}} \notag \\
        &\quad + \Big(\sum_{e\in\cE_h(\Sigma)} h_K^{-\frac{1}{2}}\|\nabla(\chi-\rI_{\mathrm{S},h}^K\chi)\cdot \bn_\Sigma\|_{0,e}\Big) \|\psi_h\|_{1,\Omega_{\mathrm{D}}}  + \sum_{K\in\cT_h(\Omega_{\mathrm{D}})}\left[ G_h^K - G^K \right](\psi_h) \bigg).
    \end{align*}
    We now apply the supremum over all $\xi_h\in\Xi_h$ (resp. $\psi_h\in \Psi_h$) to arrive at
    \begin{align*}
        \|\check{\rF}_h(\xi_h)\|_{\Xi_h'} &\lesssim \max\left\{1,\mu,\frac{\mu\alpha}{\sqrt{\kappa}}\right\} \bigg(\sum_{K\in\cT_h(\Omega_{\mathrm{S}})} \left[ |\chi-\rI_{\mathrm{S},h}^K\chi|_{2,K} + |\chi-\Pi_2^{\nabla^2,K}\chi_h|_{2,K} \right]
          \notag \\
        &\quad + \sum_{e\in\cE_h(\Sigma)} h_K^{-\frac{1}{2}}\|\nabla(\chi-\rI_{\mathrm{S},h}^K\chi)\cdot \bn_\Sigma\|_{0,e} + \sum_{e\in\cE_h(\Sigma)} h_K^{-\frac{1}{2}}\|\varphi-\rI_{\mathrm{D},h}^K\varphi\|_{0,e} \notag \\
        &\quad + \sum_{K\in\cT_h(\Omega_{\mathrm{S}})} \|F_h^K - F^K \|_{(\Xi_h(K))'}\bigg),\\
        \|\check{\rG}_h(\psi_h)\|_{\Psi_h'} &\lesssim \max\left\{1,\kappa\right\} \bigg(\sum_{K\in\cT_h(\Omega_{\mathrm{D}})}\left[|\varphi-\rI_{\mathrm{D},h}^K\varphi|_{1,K} + |\varphi-\Pi^{\nabla,K}_1\varphi_h|_{1,K}\right] \notag \\
        &\quad + \sum_{e\in\cE_h(\Sigma)} h_K^{-\frac{1}{2}}\|\nabla(\chi-\rI_{\mathrm{S},h}^K\chi)\cdot \bn_\Sigma\|_{0,e}  + \sum_{K\in\cT_h(\Omega_{\mathrm{D}})} \|G_h^K - G^K \|_{(\Psi_h(K))'} \bigg).
    \end{align*}
    Hence, the local trace inequality in \cite[Lemma 6.4]{blr2017} leads to
    \begin{align*}
        \|\check{\rF}_h(\xi_h)\|_{\Xi_h'} &\lesssim \max\left\{1,\mu,\frac{\mu\alpha}{\sqrt{\kappa}}\right\} \bigg(\sum_{K\in\cT_h(\Omega_{\mathrm{S}})}\left[|\chi-\rI_{\mathrm{S},h}^K\chi|_{2,K} + |\chi-\Pi_2^{\nabla^2,K}\chi_h|_{2,K} \right]
          \notag \\
        &\quad +  \sum_{K\in\cT_h(\Omega_{\mathrm{S}})} \left[h_K^{-1}|\chi-\rI_{\mathrm{S},h}^K\chi|_{1,K} +|\chi-\rI_{\mathrm{S},h}^K\chi|_{2,K}
        \right] \\
        &\quad + \sum_{K\in\cT_h(\Omega_{\mathrm{D}})} \left[ h_K^{-1}\|\varphi-\rI_{\mathrm{D},h}^K\varphi\|_{0,K} + |\varphi-\rI_{\mathrm{D},h}^K\varphi|_{1,K} \right]\notag \\
        &\quad + \sum_{K\in\cT_h(\Omega_{\mathrm{S}})} \|F_h^K - F^K \|_{(\Xi_h(K))'}\bigg),\\
        \|\check{\rG}_h(\psi_h)\|_{\Psi_h'} &\lesssim \max\left\{1,\kappa\right\} \bigg(\sum_{K\in\cT_h(\Omega_{\mathrm{D}})} \left[ |\varphi-\rI_{\mathrm{D},h}^K\varphi|_{1,K} + |\varphi-\Pi^{\nabla,K}_1\varphi_h|_{1,K} \right]\notag \\
        &\quad + \sum_{K\in\cT_h(\Omega_{\mathrm{S}})} \left[ h_K^{-1}|\chi-\rI_{\mathrm{S},h}^K\chi|_{1,K} + |\chi-\rI_{\mathrm{S},h}^K\chi|_{2,K}\right] + \sum_{K\in\cT_h(\Omega_{\mathrm{D}})} \|G_h^K - G^K \|_{(\Psi_h(K))'} \bigg).
    \end{align*}
    The estimate follows by putting together the inequality above with the Poincar\'e inequality  and the triangle inequality
    \begin{align*}
        e_h\leq \|(\chi_h-\rI_{\mathrm{S},h}\chi,\varphi_h-\rI_{\mathrm{D},h}\varphi)\| + \|\chi-\rI_{\mathrm{S},h}\chi,\varphi-\rI_{\mathrm{D},h}\varphi\|.
    \end{align*}
    \qed
\end{proof}

We finalise this section by providing the convergence rates for the proposed VEM scheme (cf. Section~\ref{sec:vem}).
\begin{theorem}[convergence rates]\label{th:convergence_rates}
    Under the assumptions of Lemma~\ref{lem:errorForApriori}, assume further that for $\delta\in(0,1]$, $\chi\in\rH^{2+\delta}(\Omega_{\mathrm{S}})\cap\rH^2_{\Gamma_\mathrm{S}}(\Omega_{\mathrm{S}})$ and $\varphi\in\rH^{1+\delta}(\Omega_{\mathrm{D}})\cap\rH^1_\star(\Omega_{\mathrm{D}})$. Then, the following estimate holds:
    $$e_h\lesssim \max\left\{ 1,\mu,\kappa, \frac{\mu\alpha}{\sqrt{\kappa}} \right\} h^{\delta} \left(|\chi|_{2+\delta,\Omega_{\mathrm{S}}} + |\varphi|_{1+\delta,\Omega_{\mathrm{D}}} + \|\fb\|_{0,\Omega_{\mathrm{S}}}+\|g\|_{0,\Omega_{\mathrm{D}}}\right).$$
\end{theorem}
\begin{proof}
    Note that Lemma~\ref{lem:est_poly} implies that
    \begin{align*}
        \sum_{K\in\cT_h(\Omega_{\mathrm{S}})} \|F_h^K - F^K \|_{(\Xi_h(K))'} &\lesssim h^{\delta}\|\fb\|_{0,\Omega_{\mathrm{S}}} \quad \text{and} \quad
        \sum_{K\in\cT_h(\Omega_{\mathrm{D}})} \|G_h^K - G^K \|_{(\Psi_h(K))'} \lesssim h^{\delta}\|g\|_{0,\Omega_{\mathrm{D}}}.
    \end{align*}
    Thus, a direct application of Lemmas~\ref{lem:est_poly} and \ref{lem:est_inte} in Lemma~\ref{lem:errorForApriori} yield the result. 
    \qed
\end{proof}
\begin{comment}
    \begin{remark}
        Observe that the interpolation errors associated with the boundary coupling,
        $$ \sum_{K\in\cT_h(\Omega_{\mathrm{S}})} h_K^{-\frac{1}{2}}|\chi-\rI_{\mathrm{S},h}^K\chi|_{1,K} + h_K^{\frac{1}{2}}|\chi-\rI_{\mathrm{S},h}^K\chi|_{2,K} \quad \text{and} \quad \sum_{K\in\cT_h(\Omega_{\mathrm{D}})} h_K^{-\frac{1}{2}}\|\varphi-\rI_{\mathrm{D},h}^K\varphi\|_{0,K} + h_K^{\frac{1}{2}}\|\varphi-\rI_{\mathrm{D},h}^K\varphi\|_{1,K},$$
        exhibit an additional convergence order of $h^{\delta + \frac{1}{2}}$ (cf. Lemma~\ref{lem:est_inte}). As a result, these terms are of higher order and do not affect the linear behaviour of the estimate in Theorem~\ref{th:convergence_rates}.
    \end{remark}
\end{comment}

\begin{figure}[h!]
    \centering  
    \includegraphics[width=\textwidth,trim={0.5cm 0.cm 1.5cm 0.cm},clip]{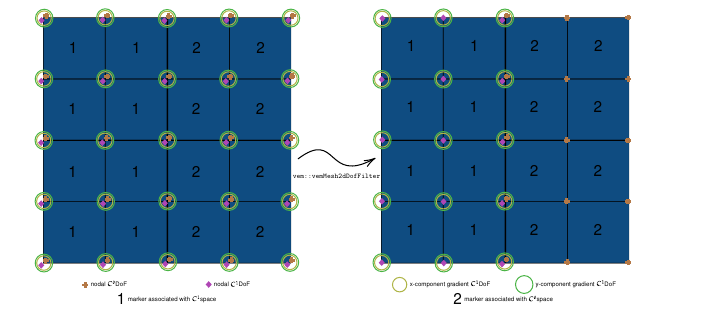}
    \caption{Schematic overview of the filtering procedure performed by the class \texttt{vem::vemMesh2dDofFilter}. Left: A discretisation of the unit square highlighting the DoFs for both $C^1$--VEM and $C^0$VEM spaces. Right: The resulting mesh after the filtering procedure, where the interface is defined by $\Sigma=\{(x,y)\in\Omega : x=\frac{1}{2}\}$ and $\Omega=(0,1)^2$.}\label{fig:sketch_class}
\end{figure}

\section{Numerical experiments} \label{sec:numerical_examples}
This section is devoted to numerical experiments that confirm the approximation properties and theoretical predictions established in Section~\ref{sec:apriori} for the VEM scheme introduced in Section~\ref{sec:vem}.

The total computational relative energy error is defined by using the local polynomial approximation of the solution as follows
\begin{align*}
    \overline{\mathrm{e}}_h^2 = \overline{\mathrm{e}}_{\chi_h}^2+\overline{\mathrm{e}}_{\varphi_h}^2 &:= %\sum_{K\in\cT_h} \frac{|(\chi-\Pi_2^{\nabla^2,K}\chi_h,\varphi- \Pi^{\nabla,K}_1\varphi_h)|_K}{|(\chi,\varphi)|_K} %\\& =
    \frac{\sum_{K\in\cT_h(\Omega_\mathrm{S})}|\chi-\Pi_2^{\nabla^2,K}\chi_h|_{2,K}^2}{\sum_{K\in\cT_h(\Omega_\mathrm{S})}|\chi|_{2,K}^2}+  \frac{\sum_{K\in\cT_h(\Omega_\mathrm{D})}|\varphi-\Pi^{\nabla,K}_1\varphi_h|_{1,K}^2}{\sum_{K\in\cT_h(\Omega_\mathrm{D})}|\varphi|_{1,K}^2}.
\end{align*}
Whereas, the order of convergence $r(\cdot)$ applied to either $\overline{\mathrm{e}}_h$, $\overline{\mathrm{e}}_{\chi_h}$, or $\overline{\mathrm{e}}_{\varphi_h}$ of the refinement $1\leq j$ is computed from the formula
$$r(\cdot)^{j+1} = \frac{\log\left(\frac{(\cdot)^{j+1}}{(\cdot)^{j}}\right)}{\log\left(\frac{h^{j+1}}{h^{j}}\right)},$$ 
with $h$ indicating the maximum element diameter   in $\cT_h$. In turn, the stabilisation term $S_\mathrm{S}^K(\chi_h,\xi_h)$ follows the ``diagonal recipe" introduced in \cite{beirao2017stab} and $S_\mathrm{D}^K(\varphi_h,\psi_h)$ is chosen as the well-known \texttt{DOFI-DOFI} stabilisation operator. We also recall that these  operators are weighted by $\mu$ and $\kappa$, respectively.

The implementation has been carried out using the object-oriented \texttt{C++} library \texttt{VEM++} \cite{dassi2023vem++}, which provides a new functionality to filter mesh DoFs according to element markers through the class \texttt{vem::vemMesh2dDofFilter}. The procedure starts from a mesh where multiple VEM spaces (e.g. the $C^1$--VEM and $C^0$--VEM spaces defined in Section~\ref{sec:C1VEM} and \ref{sec:C0VEM}, respectively) are simultaneously defined over the entire mesh. Each element is then assigned a marker identifying the space that must be retained on it. After applying the filter, only the DoFs associated with the selected space are kept on each marked element. Consequently, both sets of DoFs coexist only on the interface between regions with different markers. In Fig.~\ref{fig:sketch_class}, we provide a schematic overview of this process.

\begin{figure}[h!]
    \centering
    \subfigure[Quadrilateral. \label{fig:quad}]  {\includegraphics[width=0.24\textwidth]{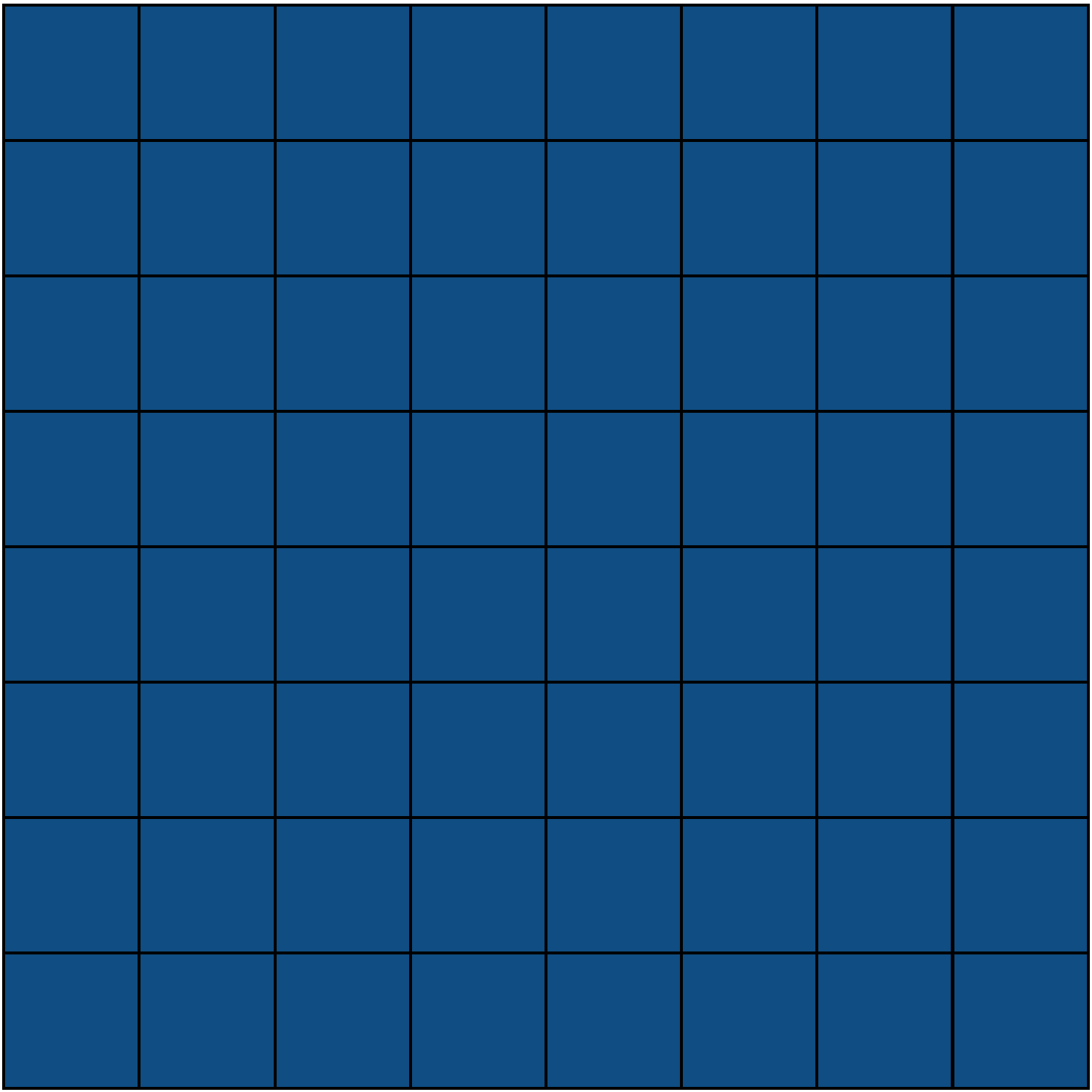}} 
    \subfigure[Non--convex. \label{fig:nonconvex}]  {\includegraphics[width=0.24\textwidth]{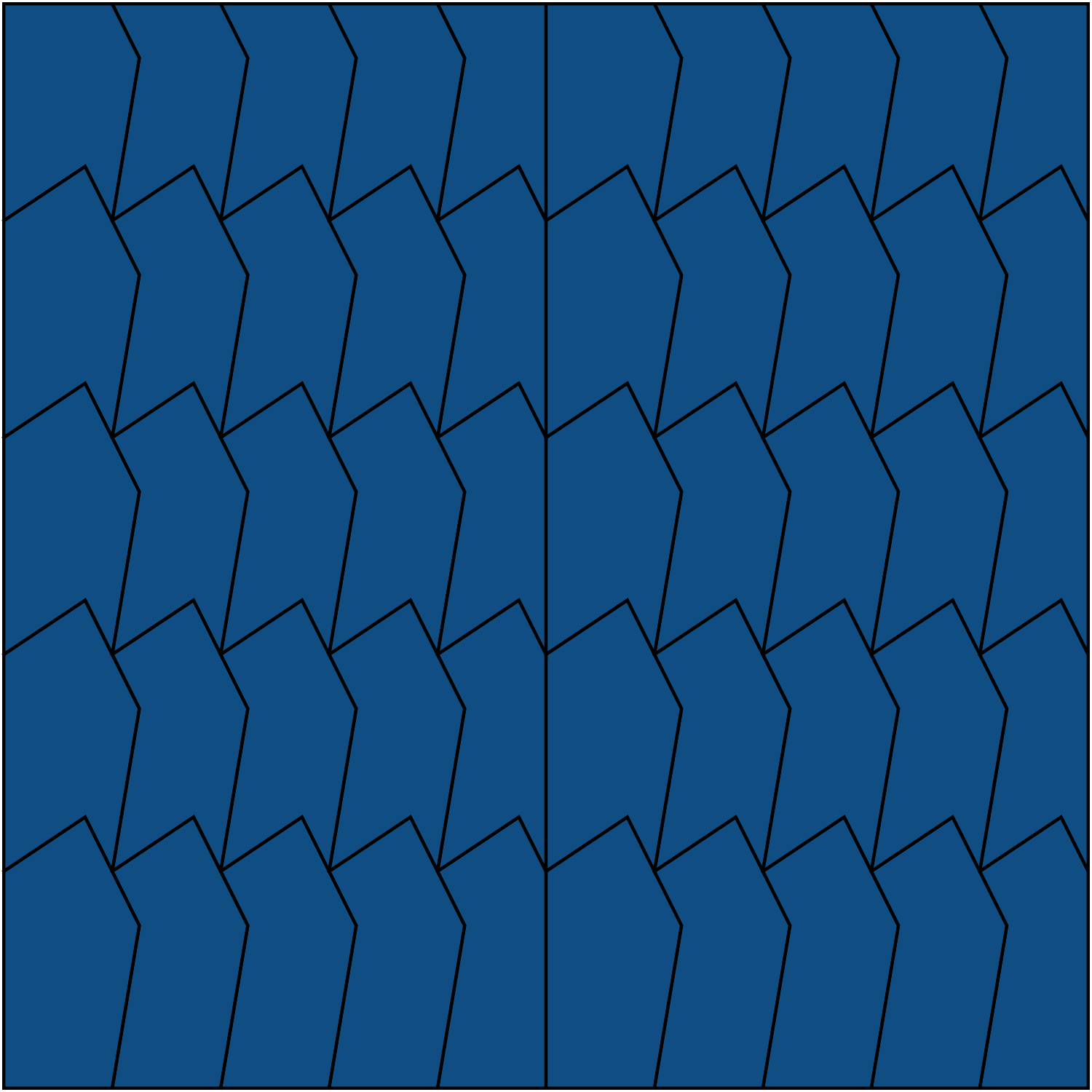}}  
    \subfigure[Voronoi. \label{fig:voro}] {\includegraphics[width=0.24\textwidth]{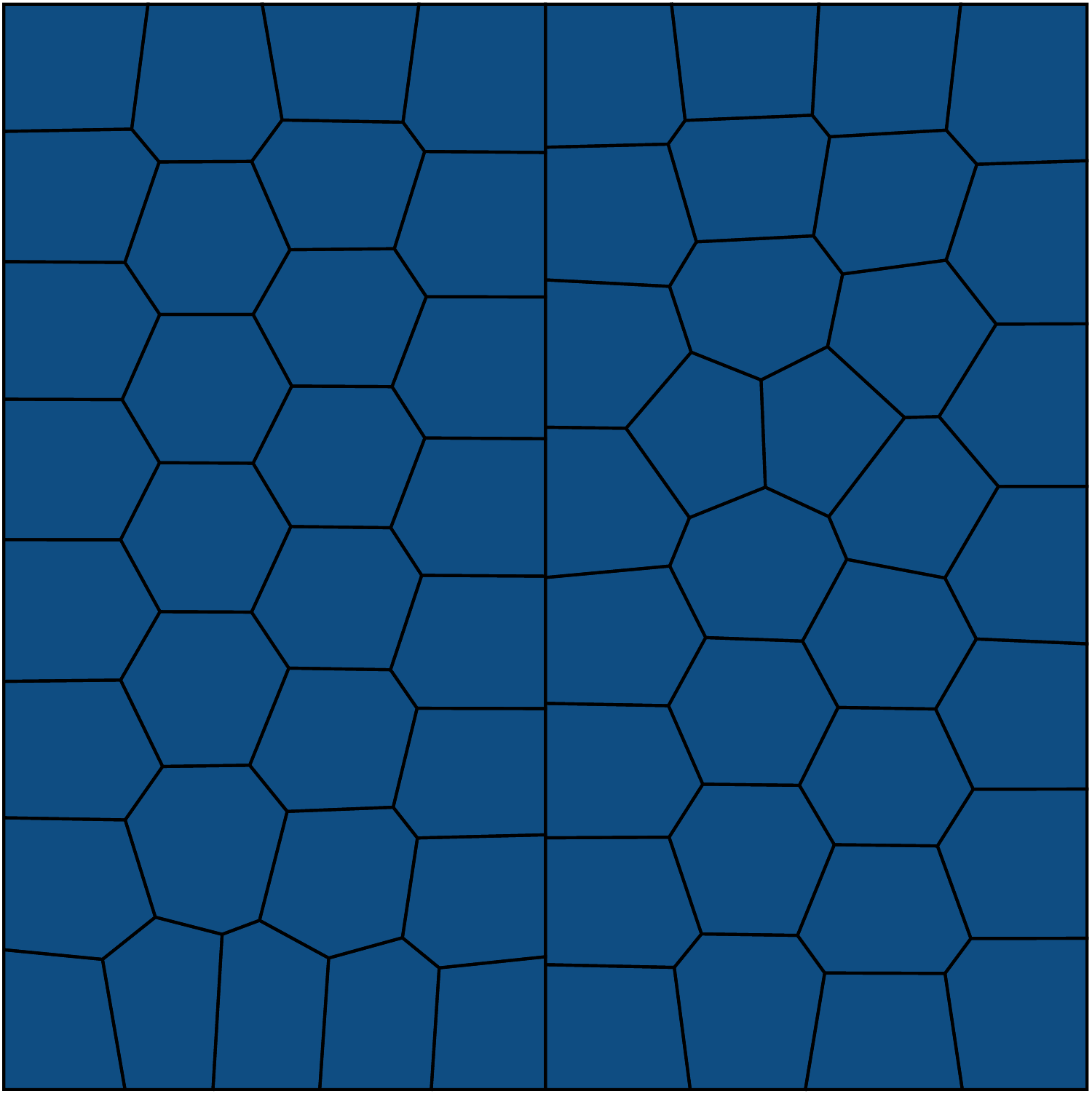}}
    \subfigure[Perturbed Voronoi. \label{fig:pert}] {\includegraphics[width=0.24\textwidth]{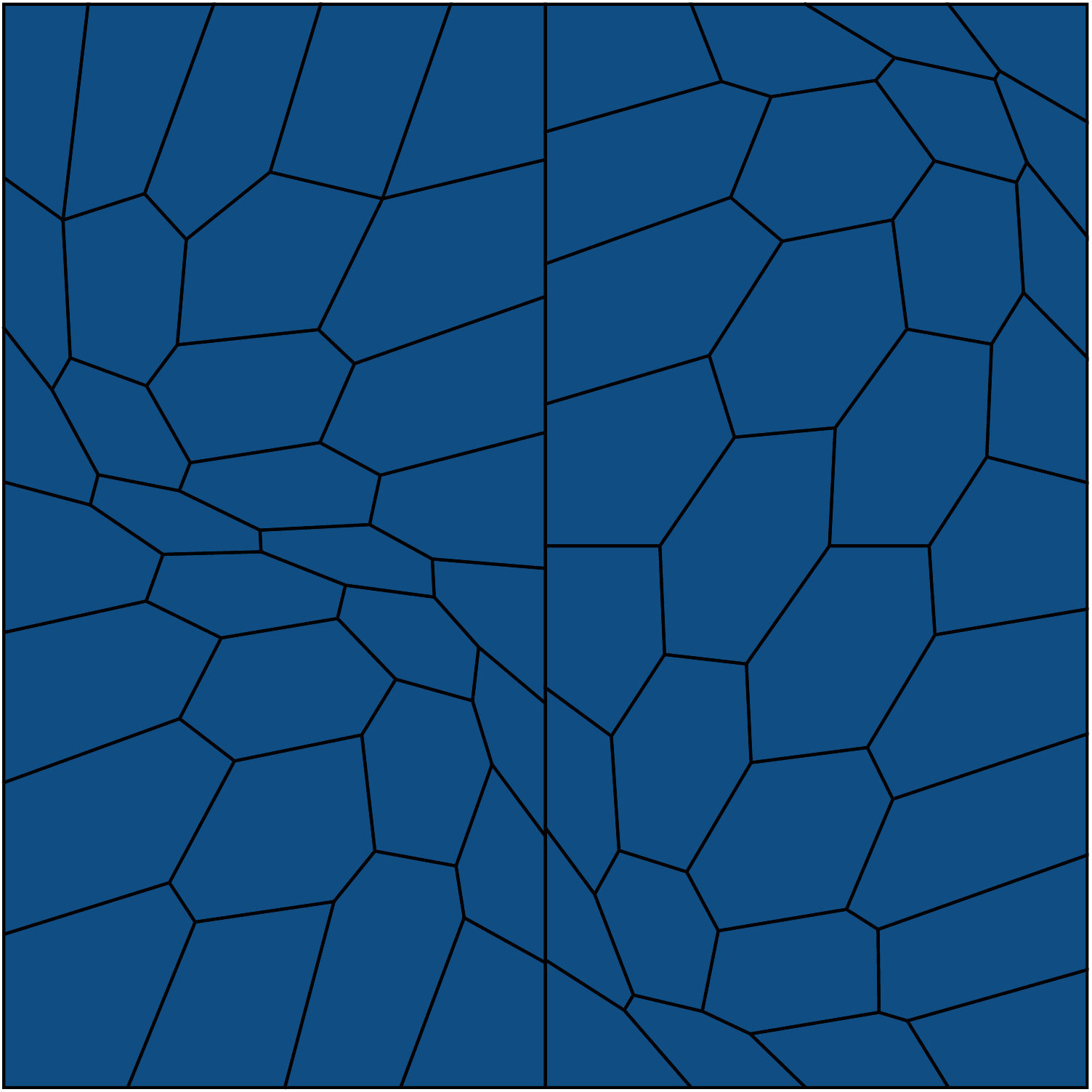}}
    \caption{Variety of discretisations used for the unit square in Experiments 1 and 2.}\label{fig:meshes}
\end{figure}

\begin{figure}[h!]
    \centering  
    \includegraphics[width=0.75\textwidth,trim={0.65cm 1.45cm 0.65cm 1.55cm},clip]{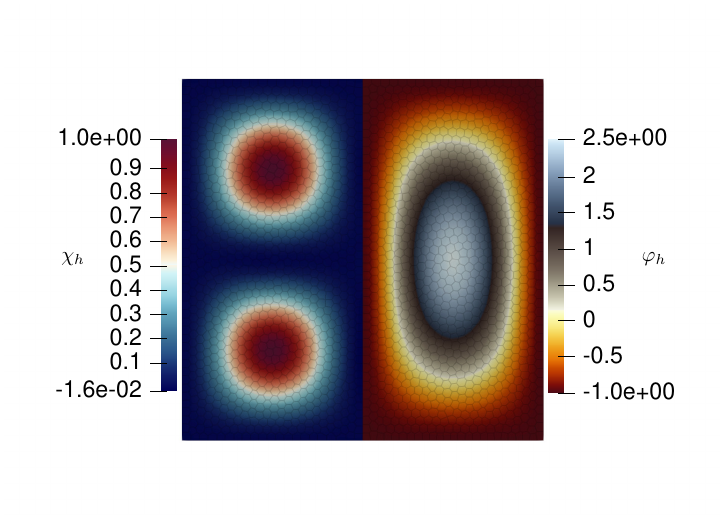}
    \caption{Experiment 1. Nodal values of the VEM solutions $\chi_h$ and $\varphi_h$ on the Voronoi mesh at the final refinement level.}\label{fig:zeroSol}
\end{figure}

\subsection{Experiment 1: manufacture solution with zero value in the interface}\label{sec:exp1}
In this experiment, we consider the domain $\Omega = (0,1)^2$ (see Fig.~\ref{fig:meshes}), which is split into two subdomains: $\Omega_{\mathrm{S}} = (0,\frac{1}{2})\times(0,1)$ and $\Omega_{\mathrm{D}} = (\frac{1}{2},1)\times(0,1)$. The interface between them is given by $\Sigma = \Gamma_{\mathrm{S}}\cap\Gamma_{\mathrm{D}} = \{(x,y)\in\Omega : x=\frac{1}{2}\}$. As illustrated in Fig.~\ref{fig:voro}-\ref{fig:pert}, the computational meshes join at $\Sigma$ with non-matching nodes, highlighting the scheme's ability to   handle nonconforming interfaces.

The smooth manufactured solutions are set as
\begin{gather*} 
    \chi(x,y) = \frac{1}{20}\sin(2\pi x)^2\sin(2\pi y)^2, \quad \varphi(x,y) = \varphi_0(x-1)^2(\frac{1}{2} - x)^2y^2(1-y)^2 - 1,\\
    \bu(x,y)=\bcurl (\chi(x,y)), \quad
    \text{and}\quad p(x,y)=0.
\end{gather*}
Note that the source terms $\fb$ and $g$ are sufficiently smooth, as they are derived from the prescribed manufactured solutions. In addition, the transmission conditions in \eqref{eq:TC} lead to additional functionals in the right-hand side of the system. The model parameters are set to unity ($\mu=\kappa=\alpha=1$) and $\varphi_0 = 1.44\times10^4$.

The error history is reported in Table~\ref{tab:convergenceZero}, where we observe an asymptotic linear decay for the   scheme  proposed in Section~\ref{sec:vem}, as predicted by Theorem~\ref{th:convergence_rates} for all the mesh types listed in Fig.~\ref{fig:meshes}. In addition, a snapshot of the nodal values for the variables of interest, which are known through the DoFs (cf. Section~\ref{sec:disc_formulation}), are shown in Fig.~\ref{fig:zeroSol} for the Voronoi mesh (see Fig.~\ref{fig:voro}) in the last refinement step.

\begin{table}[h!]
\setlength{\tabcolsep}{2pt}
 \begin{center}
\resizebox{0.75\textwidth}{!}{ 
\begin{tabular}{| c | c | c | c | c | c | c | c | c |}
\hline
{$\cT_h$} & 
{$h$} & 
{\#DoFs} & 
{$\bar{\mathrm{e}}_{h}$} & 
{$r(\bar{\mathrm{e}}_{h})$} & 
{$\bar{\mathrm{e}}_{\chi_h}$} & 
{$r(\bar{\mathrm{e}}_{\chi_h})$} & 
{$\bar{\mathrm{e}}_{\varphi_h}$} & 
{$r(\bar{\mathrm{e}}_{\varphi_h})$} \\
\hline 
\hline
\multirow{4}{*}{\rotatebox{90}{Quad.}} 
& 6.25e-02 & 613 & 3.56e-01 & $\star$    & 2.77e-01 & $\star$    & 2.23e-01 & $\star$   \\
& 3.12e-02 & 2245 & 1.80e-01 & 0.99 & 1.39e-01 & 1.00 & 1.14e-01 & 0.97\\
& 1.56e-02 & 8581 & 9.00e-02 & 1.00 & 6.94e-02 & 1.00 & 5.73e-02 & 0.99\\
& 7.81e-03 & 33541 & 4.50e-02 & 1.00 & 3.47e-02 & 1.00 & 2.87e-02 & 1.00\\
\hline
\hline
\multirow{4}{*}{\rotatebox{90}{Nonc.}} 
& 2.83e-02 & 7505 & 1.91e-01 & $\star$    & 1.70e-01 & $\star$    & 8.88e-02 & $\star$   \\
& 2.36e-02 & 10805 & 1.62e-01 & 0.91 & 1.44e-01 & 0.89 & 7.43e-02 & 0.98\\
& 2.02e-02 & 14705 & 1.41e-01 & 0.93 & 1.25e-01 & 0.91 & 6.38e-02 & 0.99\\
& 1.77e-02 & 19205 & 1.24e-01 & 0.96 & 1.10e-01 & 0.95 & 5.59e-02 & 0.99\\
\hline
\hline
\multirow{4}{*}{\rotatebox{90}{Vor.}} 
& 4.42e-02 & 2131 & 2.60e-01 & $\star$    & 2.07e-01 & $\star$   & 1.56e-01 & $\star$   \\
& 3.12e-02 & 4222 & 1.86e-01 & 0.96 & 1.48e-01 & 0.97 & 1.12e-01 & 0.95\\
& 2.21e-02 & 8347 & 1.35e-01 & 0.93 & 1.08e-01 & 0.90 & 8.01e-02 & 0.98\\
& 1.56e-02 & 16692 & 9.46e-02 & 1.02 & 7.59e-02 & 1.03 & 5.65e-02 & 1.00\\
\hline
\hline
\multirow{4}{*}{\rotatebox{90}{Pert.}} 
& 4.42e-02 & 2122 & 3.07e-01 & $\star$   & 2.52e-01 & $\star$    & 1.76e-01 & $\star$    \\
& 3.12e-02 & 4224 & 2.27e-01 & 0.88 & 1.85e-01 & 0.89 & 1.31e-01 & 0.85\\
& 2.21e-02 & 8347 & 1.63e-01 & 0.95 & 1.35e-01 & 0.92 & 9.21e-02 & 1.02\\
& 1.56e-02 & 16659 & 1.18e-01 & 0.92 & 9.55e-02 & 0.99 & 6.98e-02 & 0.80\\
\hline
\end{tabular}
}
\end{center}
%\vspace{0.25cm}
\caption{Experiment 1. Convergence history of the proposed VEM for a variety of meshes.}
\label{tab:convergenceZero}
\end{table}

\begin{figure}[h!]
    \centering  
    \includegraphics[width=0.75\textwidth,trim={0.65cm 1.45cm 0.65cm 1.55cm},clip]{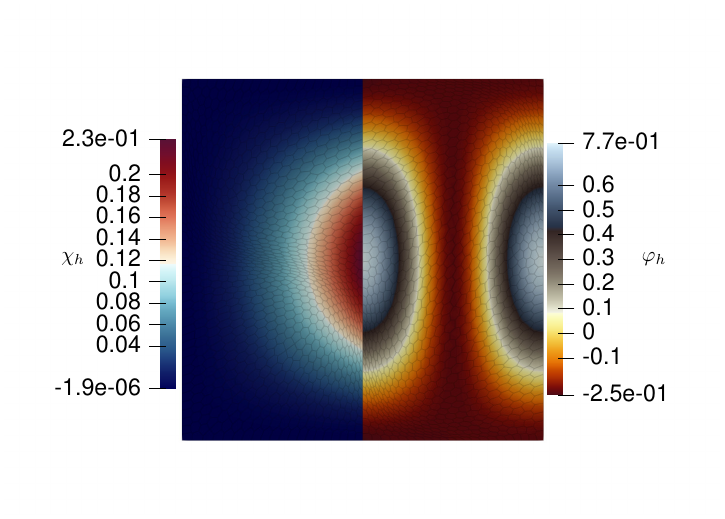}
    \caption{Experiment 2. Nodal values of the VEM solutions $\chi_h$ and $\varphi$ on the finest refinement of the perturbed Voronoi mesh.}\label{fig:nonzeroSol}
\end{figure}

\subsection{Experiment 2: manufactured solution with nonzero value in the interface}\label{sec:exp2}

We now extend the result observed in Experiment 1. Following the same domain configuration and unit parameter settings (cf. Section~\ref{sec:exp1}), we adopt a manufactured solution that does not satisfy the transmission conditions \eqref{eq:TC}. Consequently, additional right-hand side forcing terms are introduced to compensate for this mismatch. Such a manufactured solutions are given by
\begin{gather*} 
    \chi(x,y) = \sin(x)^2\sin(\pi y)^2, \quad \varphi(x,y) = \sin(\pi y)^2\cos(2\pi x)^2 - \frac{1}{4},\\
    \bu(x,y)=\bcurl (\chi(x,y)), \quad
    \text{and}\quad p(x,y) = \sin(\pi y)\cos(2 \pi x).
\end{gather*}

Table~\ref{tab:convergenceNonzero} reports the error history for the VEM scheme proposed in Section~\ref{sec:vem}. Once again, the expected linear order of convergence is observed (cf. Theorem~\ref{th:convergence_rates}) for all polygonal meshes considered (see Fig.~\ref{fig:meshes}). Moreover, the nodal values of the VEM solutions $\chi_h$ and $\varphi_h$ are shown in Fig.~\ref{fig:nonzeroSol} for the Perturbed Voronoi mesh (see Fig.~\ref{fig:pert}) in the last refinement step.

\begin{table}[h!]
\setlength{\tabcolsep}{2pt}
\begin{center}
\resizebox{0.75\textwidth}{!}{ 
\begin{tabular}{| c | c | c | c | c | c | c | c | c |}
\hline
{$\cT_h$} & 
{$h$} & 
{\#DoFs} & 
{$\bar{\mathrm{e}}_{h}$} & 
{$r(\bar{\mathrm{e}}_{h})$} & 
{$\bar{\mathrm{e}}_{\chi_h}$} & 
{$r(\bar{\mathrm{e}}_{\chi_h})$} & 
{$\bar{\mathrm{e}}_{\varphi_h}$} & 
{$r(\bar{\mathrm{e}}_{\varphi_h})$} \\
\hline 
\hline
\multirow{4}{*}{\rotatebox{90}{Quad.}} 
& 6.25e-02 & 613 & 2.56e-01 & $\star$   & 1.26e-01 & $\star$  & 2.23e-01 & $\star$  \\
& 3.12e-02 & 2245 & 1.29e-01 & 0.99 & 6.30e-02 & 1.00 & 1.12e-01 & 0.99\\
& 1.56e-02 & 8581 & 6.44e-02 & 1.00 & 3.15e-02 & 1.00 & 5.62e-02 & 1.00\\
& 7.81e-03 & 33541 & 3.22e-02 & 1.00 & 1.57e-02 & 1.00 & 2.81e-02 & 1.00\\
\hline
\hline
\multirow{4}{*}{\rotatebox{90}{Nonc.}} 
& 2.83e-02 & 7505 & 1.35e-01 & $\star$   & 1.04e-01 & $\star$    & 8.65e-02 & $\star$  \\
& 2.36e-02 & 10805 & 1.12e-01 & 1.03 & 8.61e-02 & 1.05 & 7.21e-02 & 1.00\\
& 2.02e-02 & 14705 & 9.53e-02 & 1.07 & 7.26e-02 & 1.11 & 6.18e-02 & 1.00\\
& 1.77e-02 & 19205 & 8.24e-02 & 1.09 & 6.22e-02 & 1.16 & 5.41e-02 & 1.00\\
\hline
\hline
\multirow{4}{*}{\rotatebox{90}{Vor.}} 
& 4.42e-02 & 2131 & 1.91e-01 & $\star$   & 1.12e-01 & $\star$  & 1.55e-01 & $\star$  \\
& 3.12e-02 & 4222 & 1.39e-01 & 0.91 & 8.51e-02 & 0.79 & 1.10e-01 & 0.97\\
& 2.21e-02 & 8347 & 9.76e-02 & 1.03 & 5.74e-02 & 1.13 & 7.89e-02 & 0.97\\
& 1.56e-02 & 16692 & 6.96e-02 & 0.97 & 4.23e-02 & 0.88 & 5.53e-02 & 1.02\\
\hline
\hline
\multirow{4}{*}{\rotatebox{90}{Pert.}} 
& 4.42e-02 & 2122 & 2.22e-01 & $\star$  & 1.35e-01 & $\star$  & 1.77e-01 & $\star$  \\
& 3.12e-02 & 4224 & 1.60e-01 & 0.95 & 9.63e-02 & 0.97 & 1.28e-01 & 0.93\\
& 2.21e-02 & 8347 & 1.11e-01 & 1.06 & 6.71e-02 & 1.04 & 8.83e-02 & 1.07\\
& 1.56e-02 & 16659 & 8.16e-02 & 0.88 & 5.03e-02 & 0.83 & 6.43e-02 & 0.91\\
\hline
\end{tabular}
}
\end{center}
%\vspace{0.25cm}
\caption{Experiment 2. Convergence history of the proposed VEM for a variety of meshes.}
\label{tab:convergenceNonzero}
\end{table}

\subsection{Experiment 3: quarter annulus dead-end filter} 
One important application involving coupled free-flow and porous flow is the dead-end filtration process. The design of these filtration systems plays a significant role in several industries, including the chemical, pharmaceutical, and aeronautical sectors \cite{Hanspal2006}. 

Following \cite{Liu2019}, we consider the computational domain shown in Figure~\ref{fig:quarterAnnulusDomain} where the free-flow region is defined as $\Omega_\mathrm{S} = \{ (x, y) \in \mathbb{R}^2 : 4 < x^2 + y^2 < 9, \, x > 0, \, y > 0 \}$, while the porous medium occupies the region $\Omega_\mathrm{D} = \{ (x, y) \in \mathbb{R}^2 : 1 < x^2 + y^2 < 4, \, x > 0, \, y > 0 \}$, and the interface separating both regions is given by $\Sigma = \{ (x, y) \in \mathbb{R}^2 : x^2 + y^2 = 4, \, x \geq 0, \, y \geq 0 \}$. 

The boundary conditions imposed on the viscous fluid are prescribed in terms of the velocity field, as follows
$$\bu(x,y) = 
\begin{cases} 
\left(-\frac{x}{30}, -\frac{y}{30}\right) & \text{on } \Gamma_{\mathrm{S},1}, \\[1ex]
(-\frac{1}{10}, 0) & \text{on } \Gamma_{\mathrm{S},2}, \\[1ex]
(0, -\frac{1}{10}) & \text{on } \Gamma_{\mathrm{S},3}.
\end{cases}$$
Here, the boundary sections are given by the sets $\Gamma_{\mathrm{S},1} = \{ (x, y) \in \mathbb{R}^2 : x^2 + y^2 = 9, \, x \geq 0, \, y \geq 0 \}$, $\Gamma_{\mathrm{S},2} = \{ (x, y) \in \mathbb{R}^2 : 2 \leq x \leq 3, \, y = 0 \}$, and $\Gamma_{\mathrm{S},3} = \{ (x, y) \in \mathbb{R}^2 : 2 \leq y \leq 3, \, x = 0 \}$. Notice that from the definition of a stream function, we have the relation $\bu=\bcurl\chi$. Thus,
\[
\chi(x,y) = 
\begin{cases} 
\frac{3}{10}\arctan\left(\frac{y}{x}\right) & \text{on } \Gamma_{\mathrm{S},1}, \\[1ex]
0 & \text{on } \Gamma_{\mathrm{S},2}, \\[1ex]
\frac{3\pi}{20}  & \text{on } \Gamma_{\mathrm{S},3}.
\end{cases}
\qquad \text{and} \qquad
\nabla \chi(x,y) = 
\begin{cases} 
\left(-\frac{y}{30},\frac{x}{30}\right) & \text{on } \Gamma_{\mathrm{S},1}, \\[1ex]
(0, \frac{1}{10}) & \text{on } \Gamma_{\mathrm{S},2}, \\[1ex]
(-\frac{1}{10}, 0) & \text{on } \Gamma_{\mathrm{S},3}.
\end{cases}
\]
prescribe the boundary conditions for the stream function on $\Gamma_{\mathrm{S},1}\cup\Gamma_{\mathrm{S},2}\cup\Gamma_{\mathrm{S},1}$. On the other hand, no flux boundary conditions are prescribed on $\Gamma_{\mathrm{D},2} = \{ (x, y) \in \mathbb{R}^2 : 1 \leq x \leq 2, \, y = 0 \}\cup\{ (x, y) \in \mathbb{R}^2 : 1 \leq y \leq 2, \, x = 0 \}$ and zero Darcy's pressure is given on $\Gamma_{\mathrm{D},1} = \{ (x, y) \in \mathbb{R}^2 : x^2 + y^2 = 1, \, x \geq 0, \, y \geq 0 \}$. Moreover, thanks to the mixed boundary conditions imposed for $\varphi$, we seek the solution in the space $\rH^1_{\Gamma_{\mathrm{D},1}}(\Omega_{\mathrm{D}}):=\{\psi\in \rH^1: \psi=0 \text{ on } \Gamma_{\mathrm{D},1}\}$ and no Lagrange multipliers are required for this numerical test. 

\begin{figure}[h!]
    \centering  
    \includegraphics[width=0.5\textwidth]{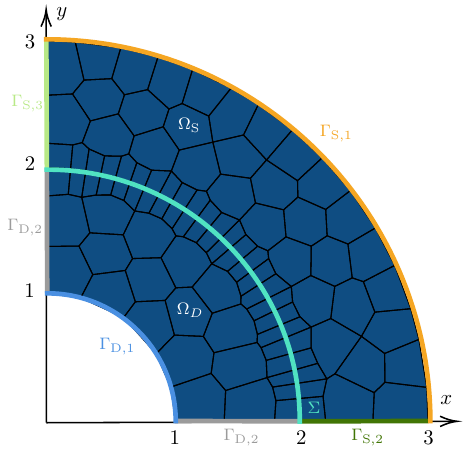}
    
    \caption{Experiment 3. Domain configuration for quarter annulus dead-end filter discretised with a coarse Voronoi mesh.}\label{fig:quarterAnnulusDomain}
\end{figure}

Snapshots of the variables of interest are provided in Fig.~\ref{fig:snapshotsQuarterAnnulus} for a Voronoi discretisation of the mesh with 15,501 elements. In particular, Fig~\ref{fig:pressureStreamQuarterAnnulusDomain} shows the nodal values of $\chi_h$ and $\varphi_h$. Whereas, Fig~\ref{fig:velocitySnapshotsQuarterAnnulus} illustrate the local polynomial approximation of the velocities as follows: $\bcurl(\Pi_2^{\nabla^2,K}\chi_h)$ for the Stokes velocity and $\kappa \nabla(\Pi_1^{\nabla,K}\varphi_h)$ for Darcy's velocity.  These fields have been smoothed by averaging the values across coincident nodes. We recall that the physical parameters for this test are set to $\mu=1$, $\kappa=10^{-2}$, and $\alpha=\frac{\sqrt{\kappa}}{\mu}$. The stream function -- pressure formulation is able to recover the results from the simulations performed previously in \cite{Liu2019}.

\begin{figure}[h!]
    \centering  
    
    \subfigure[Nodal values for $\chi_h$ and $\varphi_h$.\label{fig:pressureStreamQuarterAnnulusDomain}] {\includegraphics[width=0.4\textwidth,trim={0.675cm 0.65cm 0.49cm 0.55cm},clip]{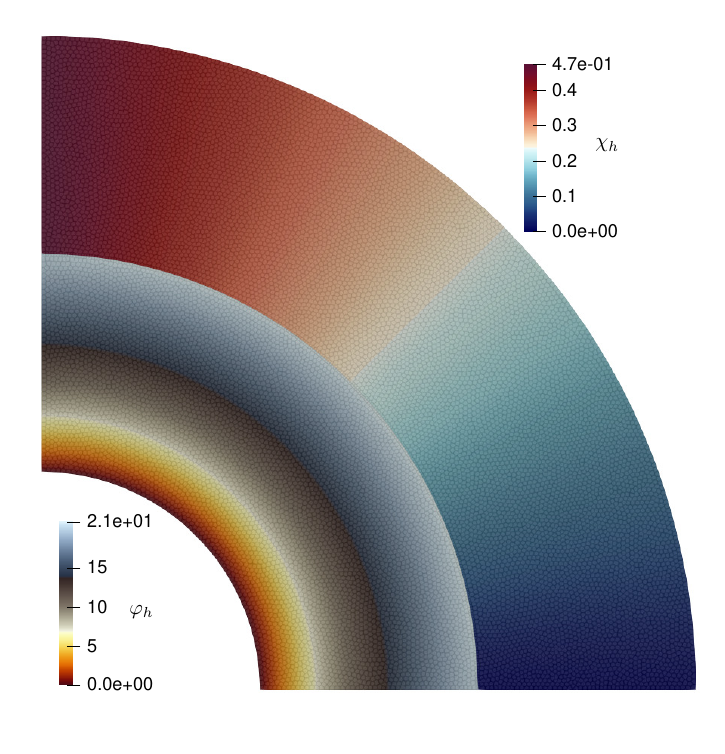}}
    \subfigure[Polynomial approximation of the velocities. \label{fig:velocitySnapshotsQuarterAnnulus}] {\includegraphics[width=0.4\textwidth,trim={5.75cm 2.25cm 4.25cm 1.15cm},clip]{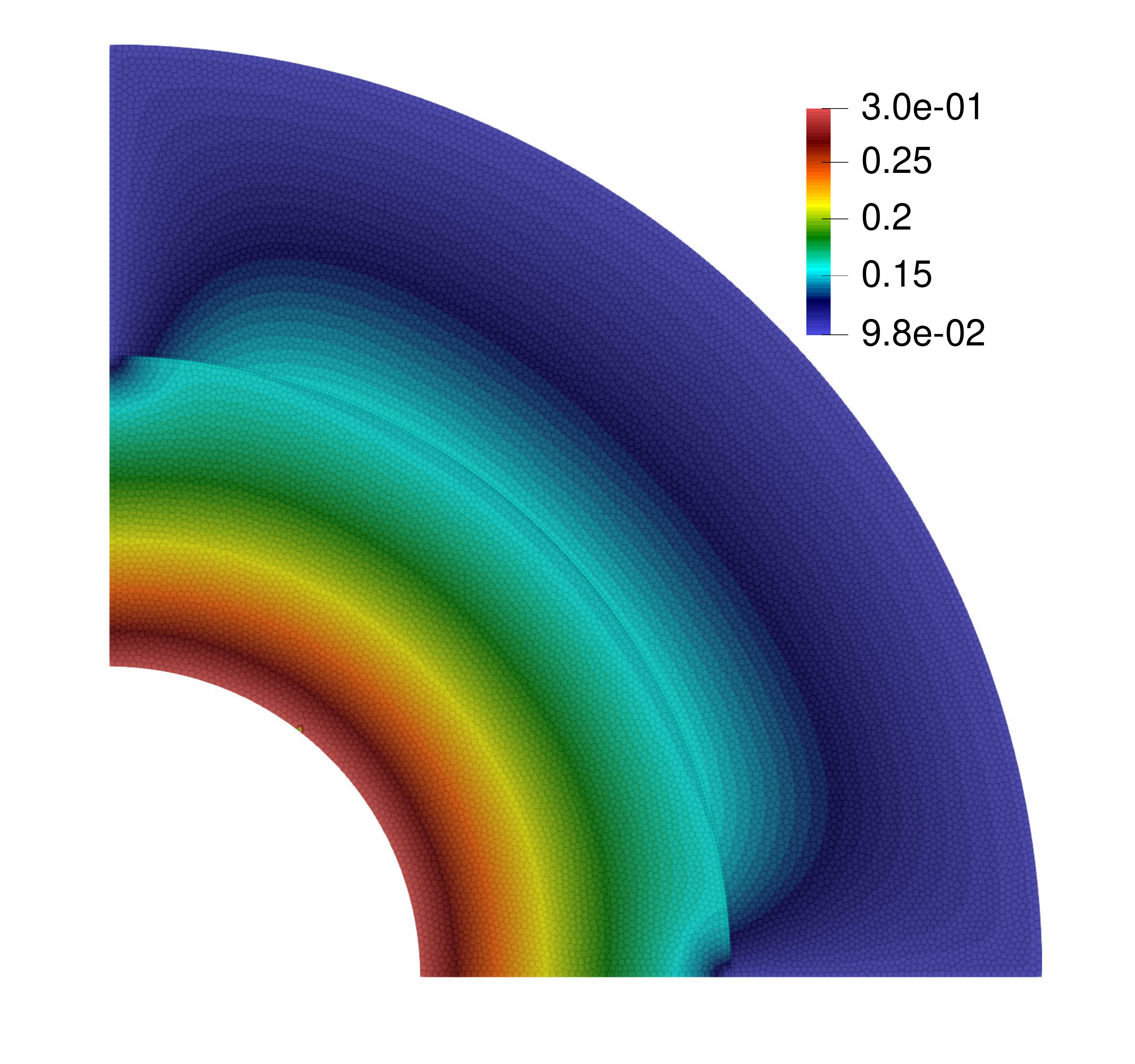}}
    
    \caption{Experiment 3. Snapshots of the variables of interest for the quarter annulus dead-end filter domain discretised with 15,501 Voronoi elements.}\label{fig:snapshotsQuarterAnnulus}
\end{figure}

\subsection{Experiment 4: blood flow network on encapsulation-based bioartificial organs}\label{sec:exp4}
Bioartificial organs are engineered constructs that combine living cells with biocompatible materials to reproduce, restore, or augment specific physiological functions. These systems are designed to interface with host tissues and operate as substitutes for damaged or missing organs. Their development typically relies on biocompatible scaffolds, which provide a structural environment for cell growth and tissue formation through the incorporation of cells and bioactive factors, we refer to \cite{Krishani23,wang19} for more details. 

A representative example is the bioartificial pancreas, where hydrogel-based scaffolds are used to encapsulate pancreatic islets containing insulin-secreting $\beta$-cells. The encapsulation device protects the pancreatic islets through filtration mechanisms provided by semi-permeable membranes \cite{dassi26,song2016}.

In this experiment, we consider a different phenomena. Following \cite{Bukac2024}, we model a simplified bioartificial organ of size $\num[round-mode=figures, round-precision=2]{0.87}  \unit{\cm}\times\num[round-mode=figures, round-precision=3]{1.37}  \unit{\cm}$; the flow network $\Omega_{\mathrm{S}}$ whose inner walls $\Sigma$ act as ultrafiltration membranes for blood plasma. The blood plasma entering to the network through $\Gamma_{\mathrm{S}}^{\mathrm{in}}$ with a velocity of $\num[round-mode=figures, round-precision=2]{3.5}  \unit{\cm}/\unit{\s}$, exits through $\Gamma_{\mathrm{S}}^{\mathrm{out}}$ with a no traction condition, and no velocity is imposed for the external wall $\Gamma_{\mathrm{S}}^{\mathrm{ext}}$. The porous biocompatible scaffold medium $\Omega_{\mathrm{D}}$ which avoids the fluid to scape through the external wall (i.e., no flux boundary conditions on $\Gamma_\mathrm{D}^{\mathrm{ext}}$) is embedded within the flow network and interacts with it through the ultrafiltration membranes $\Sigma$. Such a configuration enables the exchange of oxygen and nutrients between the blood plasma and the cells residing within the scaffold, thereby supporting the long-term viability of the transplanted cells. Fig~\ref{fig:networkFlowDomainTube} illustrates the specific domain configuration considered in this simulation, where the straight channel has a width of $\num[round-mode=figures, round-precision=1]{0.03}  \unit{\cm}$.

\begin{figure}[h!]
    \centering
    \includegraphics[width=0.5\textwidth]{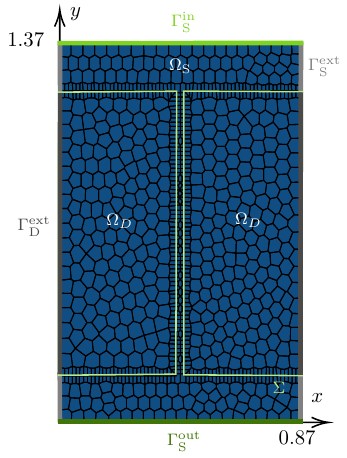}
    %\subfigure[Hexagonal channel. \label{fig:networkFlowDomainHexa}]  {\includegraphics[width=0.49\textwidth]{images/flowNetworkHexa.pdf}}  
    \caption{Experiment 4. Blood flow network with straight channel embedded within a bioartificial scaffold domain discretised with a coarse Voronoi meshes.}\label{fig:networkFlowDomainTube}
\end{figure}

Regarding the physical parameters, the blood viscosity constant is fixed to $\mu =  \num[round-mode=figures, round-precision=1]{4.e-2}  \frac{\unit{\gram}}{\unit{\cm}\, \unit{\s}}$, the slip coefficient is set to $\gamma =  \num[round-mode=figures, round-precision=2]{e3}  \frac{\unit{\g}}{\unit{\cm}^2\,\unit{\s}}$, and we explore the influence of the permeability constant by varying its values, i.e., $K \in  \{\num[round-mode=figures, round-precision=2]{e-7},\num[round-mode=figures, round-precision=2]{e-4},\num[round-mode=figures, round-precision=2]{e-1},\num[round-mode=figures, round-precision=2]{e2}\}$ with units $\frac{\unit{\cm}^3\, \unit{\s}}{\unit{\g}}$. To translate this setting to our model, we have that $\kappa = K\mu$ and $\alpha = \frac{\gamma\sqrt{\kappa}}{\mu}$. On the other hand, the stream function is set free on $\Gamma_\mathrm{S}^{\mathrm{out}}$ and inherits the following boundary conditions
\[
\chi(x,y) = 
\begin{cases} 
3.5x & \text{on } \Gamma_\mathrm{S}^{\mathrm{in}}, \\[1ex]
0 & \text{on } \Gamma_\mathrm{S}^{\mathrm{ext}}\cap \{(x,y)
\in\bbR^2 : x=0\}, \\[1ex]
3.5  & \text{on } \Gamma_\mathrm{S}^{\mathrm{ext}}\cap \{(x,y)
\in\bbR^2 : x=0.87\}.
\end{cases}
\qquad \text{and} \qquad
\nabla \chi(x,y) = 
\begin{cases} 
\left(3.5,0\right) & \text{on } \Gamma_\mathrm{S}^{\mathrm{in}}, \\[1ex]
(0, 0) & \text{on } \Gamma_{\mathrm{S}}^{\mathrm{ext}}.
\end{cases}
\]

As expected, higher values of permeability allow blood to flow more freely across the interface $\Sigma$, which consequently reduces the total volumetric flow rate remaining within the straight channel. The simulations indicate that while higher permeability values can be physically accommodated by the model, the width of the channel would need to be proportionally increased to prevent excessive fluid passing through the interface. Snapshots of the computed velocity and pressure fields are provided in Figs.~\ref{fig:velocityPlotsStraightChannel}--\ref{fig:streamPressurePlotsStraightChannel}.

\begin{figure}[h!]
    \centering
    \subfigure[$K=10^{-7}$.]  {\includegraphics[width=0.4\textwidth,trim={14.5cm 3.cm 6.cm 2.75cm},clip]{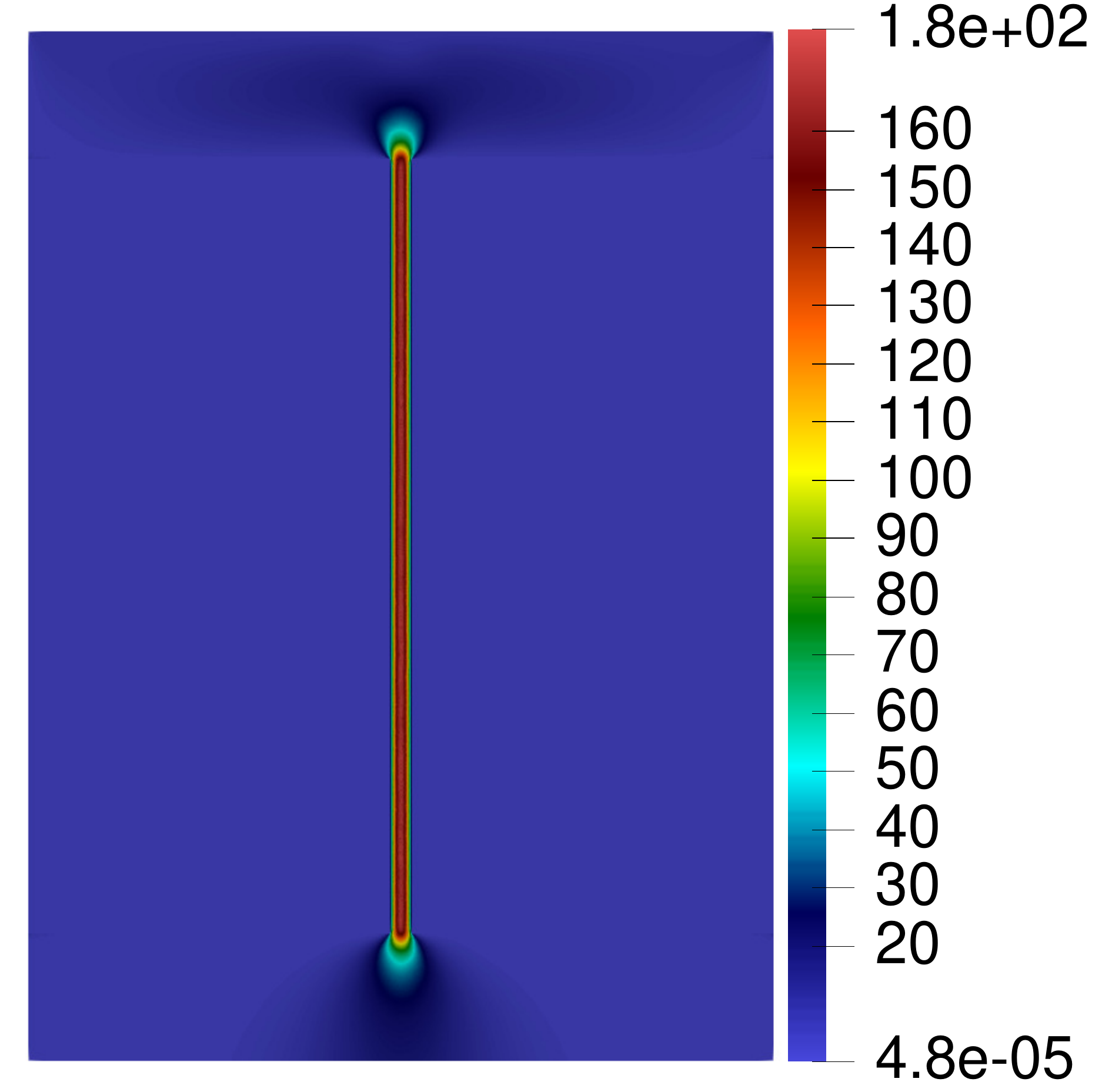}} 
    \subfigure[$K=10^{-4}$.]  {\includegraphics[width=0.4\textwidth,trim={14.5cm 3.cm 6.cm 2.75cm},clip]{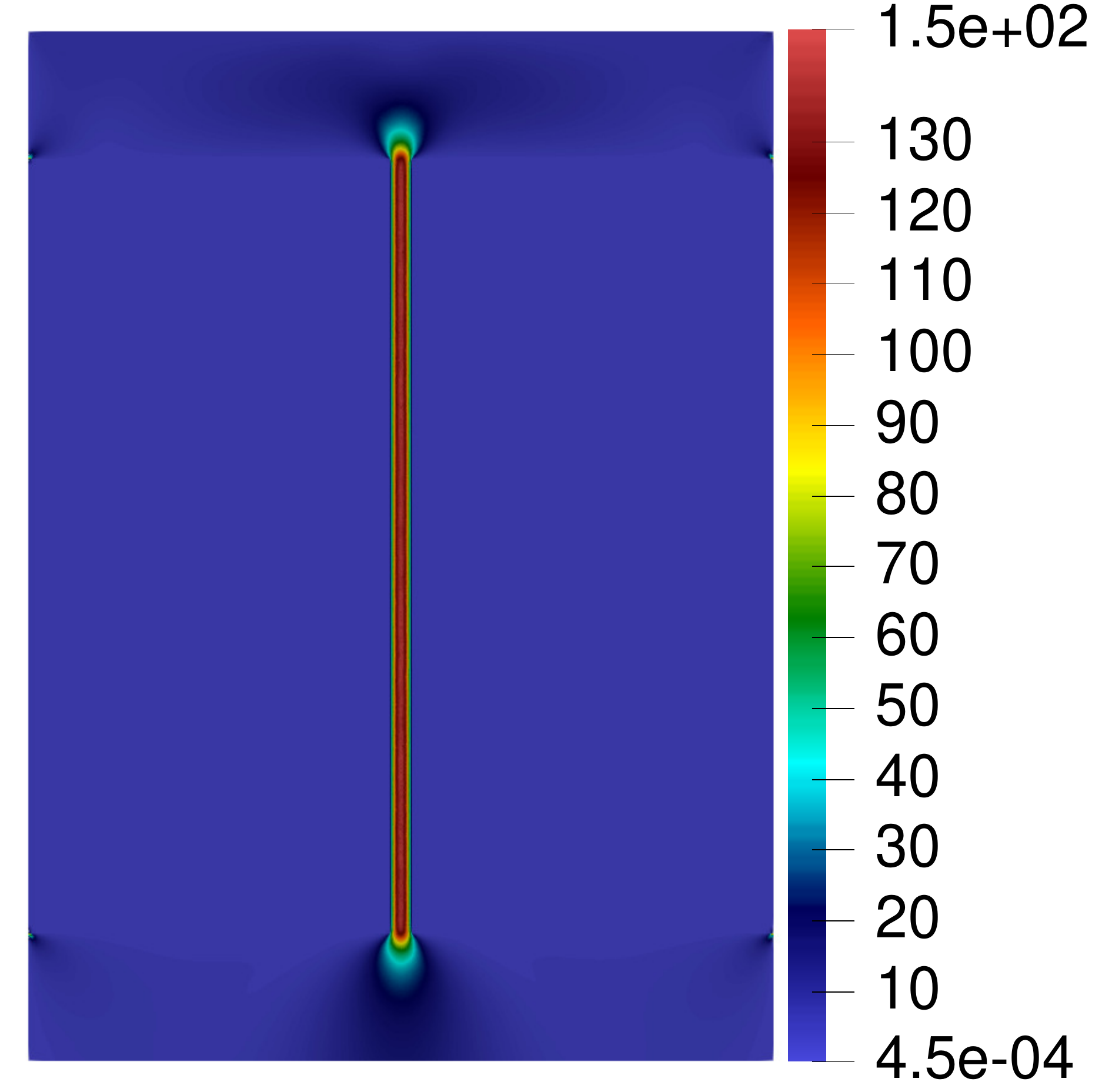}}  
    \subfigure[$K=10^{-1}$.]  {\includegraphics[width=0.4\textwidth,trim={14.5cm 3.cm 6.cm 2.75cm},clip]{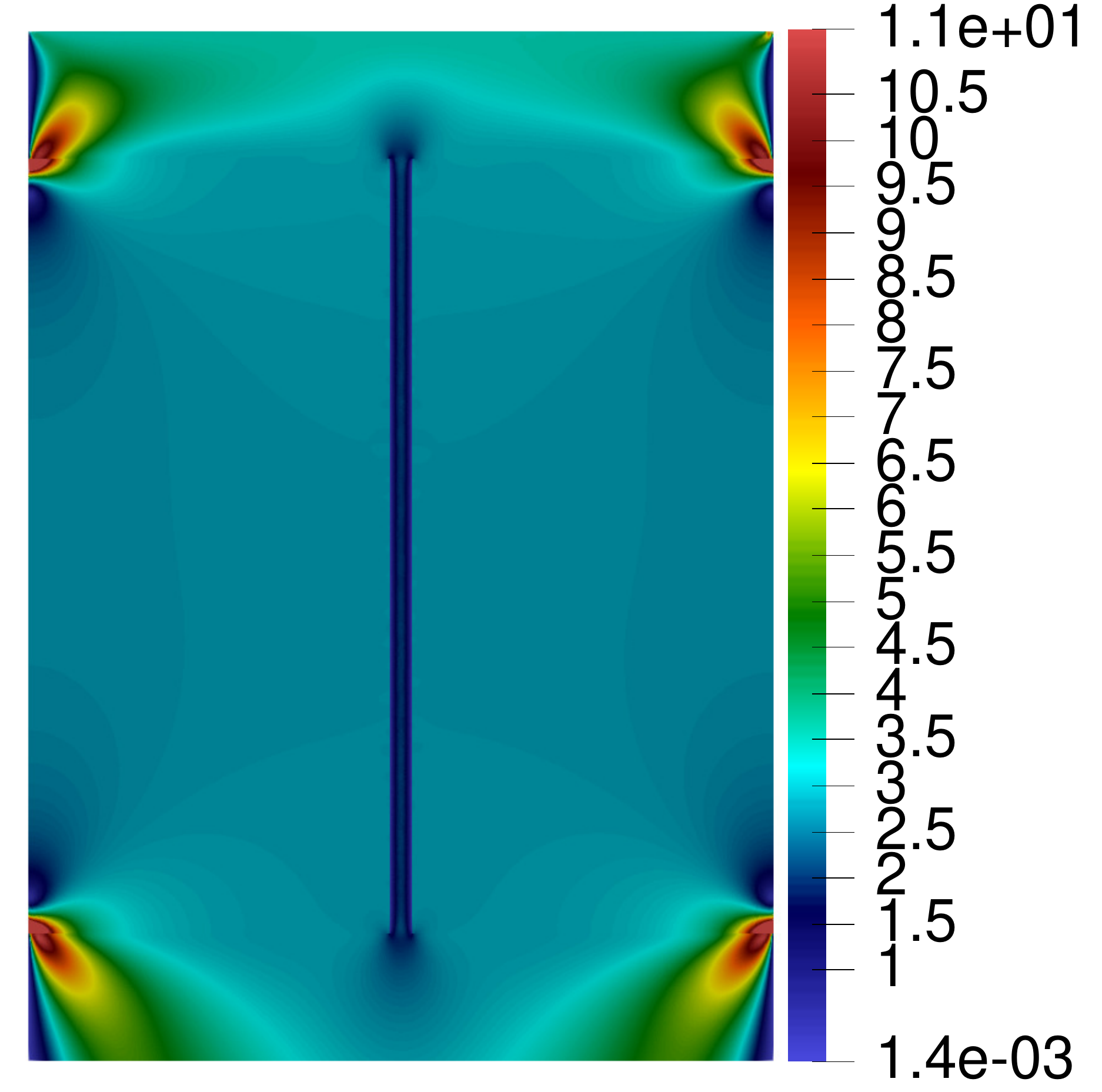}} 
    \subfigure[$K=10^{2}$.]  {\includegraphics[width=0.4\textwidth,trim={14.5cm 3.cm 6.cm 2.75cm},clip]{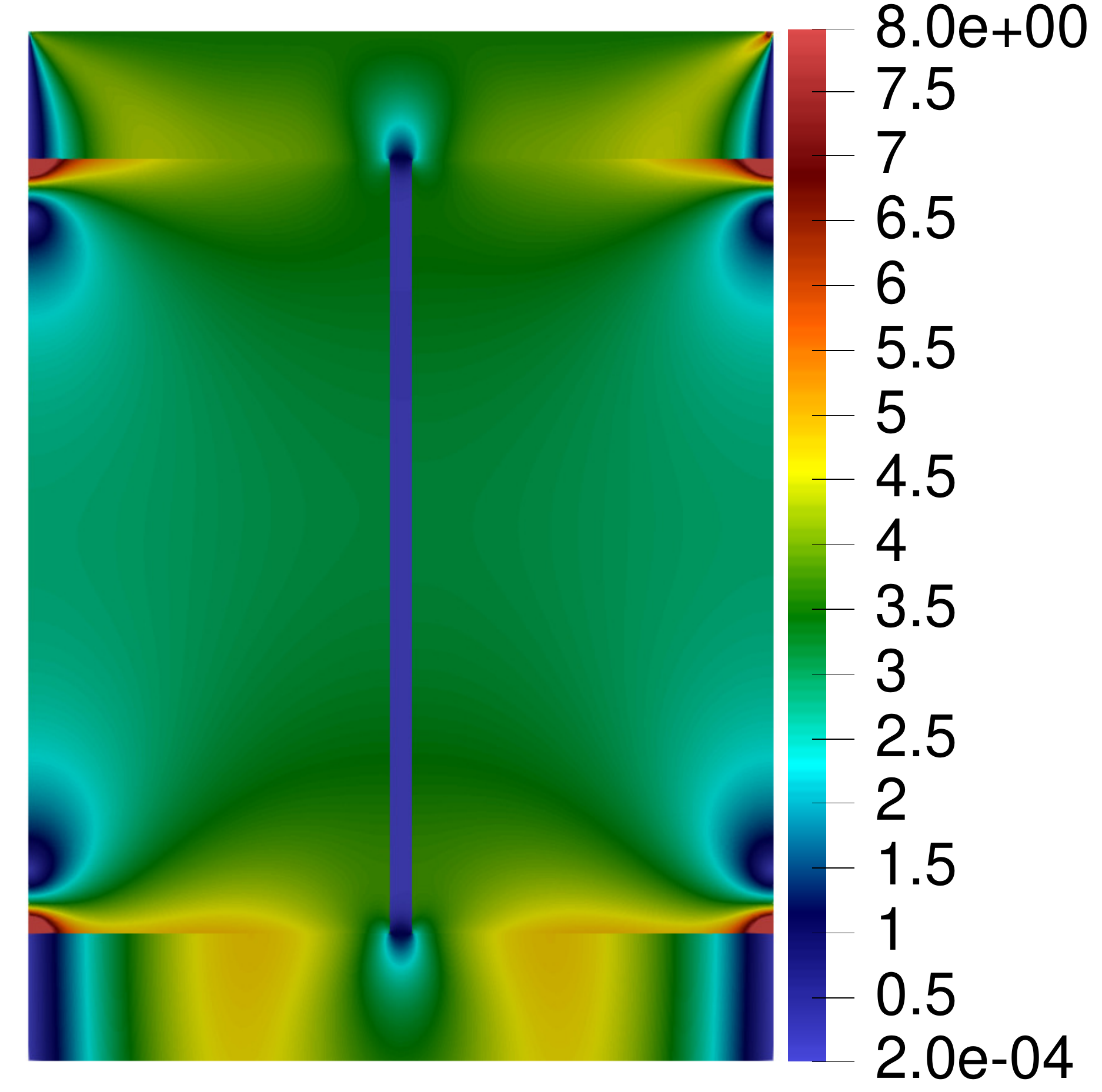}}  
    \caption{Experiment 4. Stokes' and Darcy's velocities for a straight blood flow network embedded within a bioartificial scaffold domain discretised with a Voronoi meshes composed of 55,292 elements for a variety different values of $K$.}\label{fig:velocityPlotsStraightChannel}
\end{figure}

\begin{figure}[h!]
    \centering
    \subfigure[$K=10^{-7}$.]  {\includegraphics[width=0.4\textwidth,trim={2.85cm .25cm 0.5cm .25cm},clip]{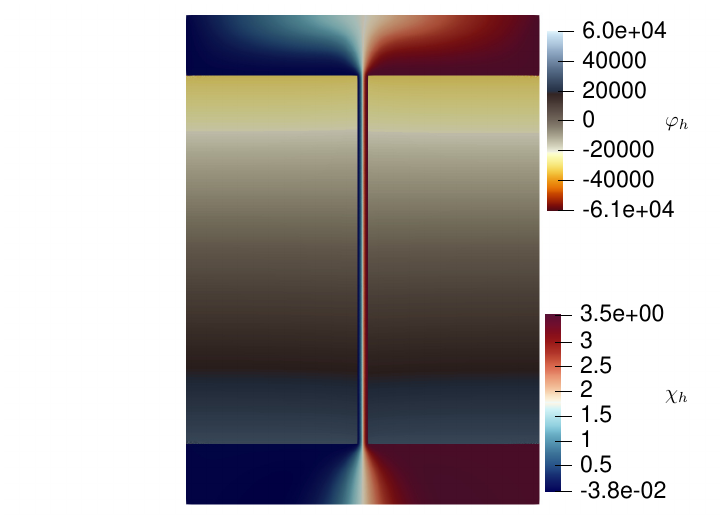}} 
    \subfigure[$K=10^{-4}$.]  {\includegraphics[width=0.4\textwidth,trim={2.85cm .25cm 0.5cm .25cm},clip]{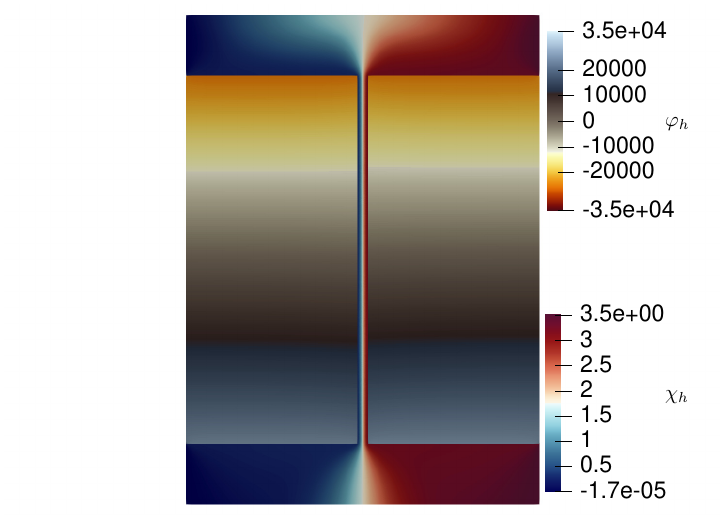}}  
    \subfigure[$K=10^{-1}$.]  {\includegraphics[width=0.4\textwidth,trim={2.85cm .25cm 0.5cm .25cm},clip]{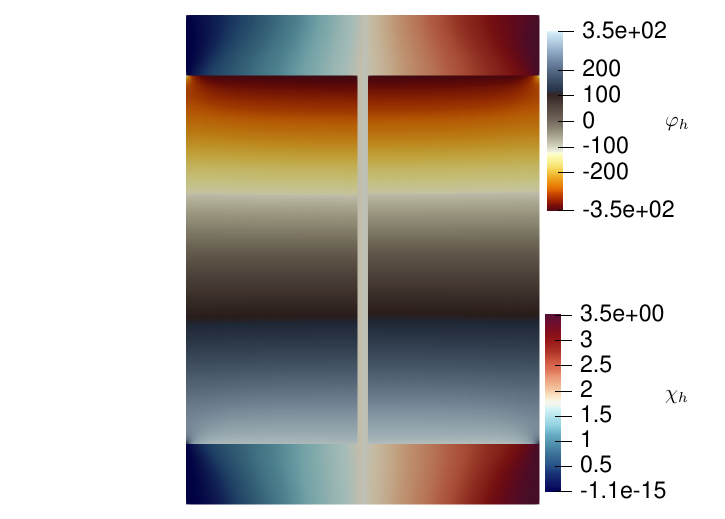}} 
    \subfigure[$K=10^{2}$.]  {\includegraphics[width=0.4\textwidth,trim={2.85cm .25cm 0.5cm .25cm},clip]{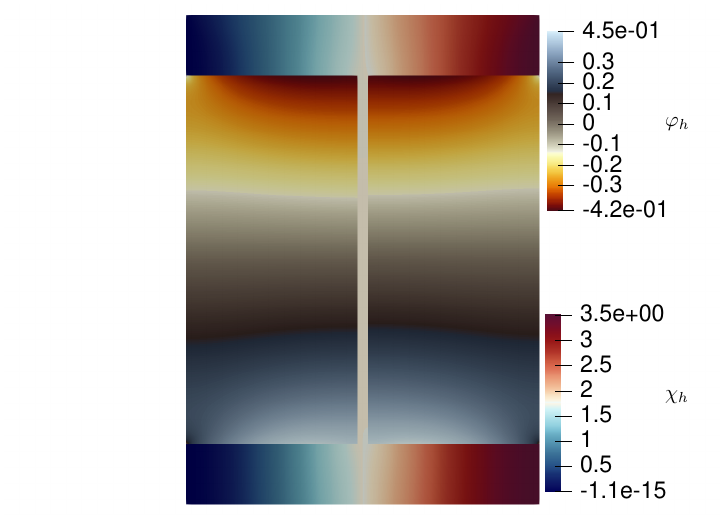}}  
    \caption{Experiment 4. Nodal values of the stream function $\chi_h$ and Darcy's pressure $\varphi_h$ for a straight blood flow network embedded within a bioartificial scaffold domain discretised with a Voronoi meshes composed of 55,282 elements for a variety different values of $K$.}\label{fig:streamPressurePlotsStraightChannel}
\end{figure}

\section{Conclusion}\label{sec:conclusion}
%Looking back at the theoretical foundation established in this work, 
Several steps in the mathematical proofs presented herein posed 
%unique and rigorous 
challenges. More specifically,   the seamless incorporation of transmission conditions into the stream–pressure formulation was non-trivial, 
in part due to the 
%a task significantly complicated by the inherent 
absence of stress tensors in the formulation. An advantage is that %\Overcoming this difficulty revealed that 
the underlying structure of both the continuous and discrete problems yields a completely inf–sup-free strategy. 
%, thereby allowing a straightforward application of the Babu\v{s}ka–Brezzi theory \rrbcom{revise. something is off here}. 
In turn, the error analysis required the precise integration of specialised interpolation and projection operators alongside trace inequalities to yield optimal convergence rates.

Despite recent advances, extending stream-function VEM formulations to multiple connected domains remains a major open challenge. In addition, scaling this framework to a three-dimensional setting necessitates transitioning to a vector potential, which inherently demands the construction of high-order $C^1$--VEM spaces for vector-valued functions. Furthermore, we are also interested in the integration of stabilisation-free techniques and the derivation of robust a posteriori error estimators, which have not yet been  addressed in this context. From a geometric and computational perspective, loosening restrictive meshing assumptions, such as explicitly permitting small edges and incorporating high-order curved interfaces, represents an important direction for future research.

%\rrb{Let's just keep the funding (I've commented the others out). If the specific journal requires to include this explicitly in the manuscript, we can put it back}
%\subsection*{Declarations}
%\textbf{Conflict of interest declaration.} The authors declare that they have no conflict of interest.

\smallskip 
\noindent\textbf{Funding.} 
FD was partially supported by the European Research Council project NEMESIS (Grant No. 101115663). DM was partially supported by ANID-Chile through FONDECYT project 1261217 and by Centro de Modelamiento Matem\'atico (CMM), FB210005,
BASAL funds for centers of excellence. AER and RRB have been partially supported by the Australian Research Council through the \textit{Future Fellowship} grant FT220100496. AER has been partially supported by  the Center for Mathematical Modeling (CMM), Proyecto Basal FB210005. RRB acknowledges the support from  the Center for Advanced Study (CAS) at the Norwegian Academy of Science and Letters under the program \textit{Mathematical Challenges in Brain Mechanics}. 

%\smallskip 
%\noindent\textbf{Author contributions.}RK, DM, AER and RRB designed the methodology. FD and AER developed the code. AER performed the experiments. DM, RRB and AER wrote the manuscript. All authors reviewed and approved the final version of the manuscript. \aer{TODO:  FD ... RK ... DM ... RRB ...}

\bibliographystyle{spmpsci}
\bibliography{bibliography}

\end{document}